\documentclass[smallcondensed]{svjour3}

\smartqed

\usepackage[utf8]{inputenc}
\usepackage[english]{babel}
\usepackage[T1]{fontenc}
\usepackage{amsmath}
\usepackage{amssymb}
\usepackage{wasysym}
\usepackage{bbm}
\usepackage{mathtools}
\usepackage{hyperref}
\usepackage{todonotes}

\usepackage{caption}
\usepackage{subcaption}

\usepackage[]{algorithm2e}
\usepackage{float}
\usepackage{stackengine}
\usepackage{diagbox}

\usepackage[square,sort,comma,numbers]{natbib}
\usepackage{fancyhdr}

\usepackage{graphicx}
\graphicspath{}

\newcommand{\NN}{\mathbb{N}}
\newcommand{\RR}{\mathbb{R}}
\newcommand{\de}{\mathrm{d}}
\newcommand{\norm}[2]{\bigl\| #1 \bigl\|_{#2}}
\newcommand{\scpr}[2]{\langle #1,#2\rangle_X}

\newcommand{\Id}{\mathrm{Id}}
\newcommand{\argmin}{\operatornamewithlimits{argmin}}
\newcommand{\p}{\mathcal{P}}

\newcommand{\Lin}{\mathcal{L}(X,X^*)}
\newcommand{\keins}{\kappa_1}
\newcommand{\kzwei}{\kappa_2}
\newcommand{\ksum}{\kappa_1 + \kappa_2}

\newcommand\xrowht[2][0]{\addstackgap[.5\dimexpr#2\relax]{\vphantom{#1}}}

\begin{document}

\title{Second order semi-smooth Proximal Newton methods in Hilbert Spaces\thanks{This work was funded by the DFG SPP 1962: Non-smooth and Complementarity-based Distributed Parameter Systems -- Simulation and Hierarchical Optimization; Project number: SCHI 1379/6-1}}

\author{Bastian P\"otzl         \and
	Anton Schiela         \and
	Patrick Jaap         
}

\institute{Bastian P\"otzl \at
	University of Bayreuth, Chair of Applied Mathematics, Universit\"atsstraße 30, 95440 Bayreuth \\
	\email{bastian.poetzl@uni-bayreuth.de} \\
	ORCID iD: https://orcid.org/0000-0002-3578-6424
	\and
	Anton Schiela \at
	University of Bayreuth, Chair of Applied Mathematics, Universit\"atsstraße 30, 95440 Bayreuth \\
	\email{anton.schiela@uni-bayreuth.de}
	\and 
	Patrick Jaap \at
	Technische Universit\"at Dresden, Institut f\"ur Numerische Mathematik, Zellescher Weg 12-14, 01069 Dresden \\
	\email{patrick.jaap@tu-dresden.de}
}

\date{Received: date / Accepted: date}

\maketitle

	
\begin{abstract}
	We develop a globalized Proximal Newton method for composite and possibly non-convex minimization problems in Hilbert spaces. Additionally, we impose less restrictive assumptions on the composite objective functional considering differentiability and convexity than in existing theory. As far as differentiability of the smooth part of the objective function is concerned, we introduce the notion of second order semi-smoothness and discuss why it constitutes an adequate framework for our Proximal Newton method. However, both global convergence as well as local acceleration still pertain to hold in our scenario. Eventually, the convergence properties of our algorithm are displayed by solving a toy model problem in function space.
\end{abstract}

\section{Introduction} \label{sec:intro}
Subject of this work is to generalize the idea of Proximal Newton methods for composite objective functions to a Hilbert space setting, aiming for the efficient solution of non-convex, non-smooth variational problems. The optimization problem reads 
\begin{align}\label{eq:prob} 
\min_{x \in X} F(x) \coloneqq f(x) + g(x)
\end{align}
where $f: X \rightarrow \RR$ is assumed to be smooth in some adequate sense and $g: X \rightarrow \RR$ is possibly not. The domain of both $f$ and $g$ is given by a subset of an arbitrary Hilbert space $X$. 

Originally, Fukushima and Mine introduced the Proximal Gradient method in the Euclidean $\RR^n$ for optimization problems of the above form, cf. \cite{FUKUSHIMA1981}. More specifically, this early version of the Proximal Gradient method constitutes a special case of a procedure studied by Tseng and Yun, cf. \cite{Tseng2007}. Further research showed that variously defined line search techniques lead to global convergence of the algorithm even under appropriate inexactness conditions for the solutions of the subproblem for step computation, cf. for example \cite{Byrd2015,Fountoulakis2018,Ghanbari2017,Lee2019,Scheinberg2016,Stella2017}. Additionally, local acceleration results have been achieved by utilizing second order information of the smooth part close to optimal solutions of the original minimization problem.

Obviously, further assumptions on the form of the composite objective functional open the door to more specific adaptions of the solution algorithm. For example in \cite{Li2016,TranDinh2015,Dinh2013}, the authors assume convexity and self-concordance of the smooth part $f$ in order to employ damped Proximal Newton methods. Alternatively, reformulations of the original minimization problem can be useful. As a consequence, methods which have been proven to work for other problem classes can also be applied in our case. For example in  \cite{Chen2016,Chen2016a,Li2014} fixed point algorithms were employed or consider \cite{Argyriou2011} for a reformulation of \eqref{eq:prob} as a constrained problem.

A different point of view onto this class of problems was taken by Milzarek and Ulbrich in \cite{Milzarek2014}. For $g(x) \coloneqq \lambda \norm{x}{1}$ with $\lambda > 0$, they considered a semi-smooth Newton method with filter globalization which Milzarek later on generalized to work also for arbitrary convex functions for $g$, cf. \cite{Milzarek2016}. 

Recently, Kanzow and Lechner discussed a globalized, inexact and possibly non-convex Proximal Newton-type method in Euclidean space $\RR^n$, cf. \cite{Kanzow2020}. There, the algorithm resorted to Proximal Gradient steps in the case of insufficient descent together with a line-search procedure in order to achieve global convergence and cope with lacking convexity of the objective functional. 

The work of Lee and Saunders \cite{Lee2014} gives an instructive overview of a generic version of the Proximal Newton method as well as several convergence results.
Our contributions beyond \cite{Lee2014} can be summarized as follows: Most obviously, we generalize the Euclidean space setting to a Hilbert space one. Additionally, in \cite{Lee2014} only elliptic bilinear forms for the second order model are considered and the non-smooth part $g$ is required to be convex. We use a more general framework of convexity assumptions for the composite objective function $F$. Furthermore, we do not demand second order differentiability with Lipschitz-continuous second order derivative of the smooth part $f$ but instead settle for adequate semi-smoothness assumptions. We replace the simple line-search approach for globalization with a more sophisticated proximal arc-search method which additionally softens the convexity assumptions on the objective functional. Eventually, we establish a more refined version of the global convergence proof and also give a dual interpretation for the stopping criterion of the algorithm. To our knowledge, also the notion of second order semi-smoothness for $f$ is yet to appear in literature. On the other hand, our work here covers neither inexact nor Proximal Quasi-Newton methods.

An important practical aspect of splitting methods, such as Proximal Newton, is that the non-smooth part $g$ of the composite objective functional $F$ yields a proximity operator $\mathrm{prox}_g$ that can be evaluated easily. This is, for example, the case, if $g$ and also the employed scalar product have diagonal structure. Then the solution of the subproblem within the proximity operator can be computed cheaply in a componentwise fashion. In function space problems, in particular if Sobolev spaces are involved, it is known that instead of a diagonal structure, a multi-level structure should be used in order to reflect the topology of the function space properly. Diagonal proximal operators would suffer from mesh-dependent condition numbers. In our numerical computations we therefore employ non-smooth multi-grid techniques to compute the Proximal Newton steps, in particular Truncated Non-smooth Newton Multigrid Methods, cf. \cite{Graeser2018}.

Let us first specify the setting in which we will discuss the convergence properties of Proximal Newton methods in a real Hilbert space $(X,\langle\cdot,\cdot\rangle_X)$ with corresponding norm $\|v\|_X=\sqrt{\langle v,v\rangle_X}$ and dual space $X^*$. The Hilbert space structure of $X$ also gives us access to the Riesz-Isomorphism $\mathcal R : X \to X^*$, defined by $\mathcal Rx=\langle x,\cdot\rangle_X$, which satisfies $\norm{\mathcal{R} x}{X^*} = \norm{x}{X}$ for every $x\in X$. Since $\mathcal R$ is non-trivial in general, we will not identify $X$ and $X^*$. 

We will assume the smooth part of our objective functional $f:X \rightarrow \RR$ to be continuously differentiable with Lipschitz-continuous derivative $f':X\to X^*$, i.e., we can find some constant $L_f > 0$ such that for every $x,y \in X$ the estimate
\begin{align} \label{eq:Lipschitz}
\norm{f'(x) - f'(y)}{X^*} \leq L_f \norm{x-y}{X}
\end{align}
holds.

Next we will specify our assumptions on the second order model for $f$. 
In what follows, we will notationally identify the linear operators $H_x \in \Lin$ with the corresponding symmetric bilinear form $H_x : X \times X\to \RR$, and write $(H_x v)(w)=H_x(v,w)$, using the abbreviation $H_x(v)^2=H_x(v,v)$.  We will assume uniform boundedness of $H_x$ along the sequence $(x_k)$ of iterates:
\[
\exists M \in \RR : \|H_{x_k}\|_{\Lin} \le M \quad \forall k \in \NN.  
\]
In addition, along the sequence of iterates $x_k$ we assume a uniform bound of the form
\begin{align} \label{eq:kappa1}
\exists \keins \in \RR : H_{x_k} (v)^2 \coloneqq H_{x_k}(v,v) \geq \keins \norm{v}{X}^2 \quad \forall v\in X, k \in \NN. 
\end{align}
For $\keins > 0$ estimate \eqref{eq:kappa1} represents ellipticity of $H_x$ with constant $\keins$. When considering exact (and smooth) Proximal Newton methods, where $H_x$ is given by the second-order derivative of $f$ at some point $x \in X$, \eqref{eq:kappa1} is equivalent to $\keins$-strong convexity of $f$. In the case $\keins > 0$ we may also define an energy-norm and write:
\[
\norm{v}{H_x}^2 \coloneqq H_x(v,v).
\]
For most of the paper we may choose $H_x$ freely in the above framework. For fast local convergence, however, we will impose a semi-smoothness assumption, cf. \eqref{eq:semismooth}. Semi-smooth Newton methods in function space have been discussed, for example, in   \cite{Ulbrich2002,Ulbrich2011,Hintermueller2002,Schiela2006}.
Furthermore, in order to guarantee transition of our globalization scheme to fast local convergence, we suppose $f$ to suffice the notion of second order semi-smoothness (cf. Section~\ref{sec:soss}) which generalizes second order differentiability in our setting and the definition of which slightly differs from semi-smoothness of $f'$ in \eqref{eq:semismooth}. 

We assume that the non-smooth part $g$ is lower semi-continuous and satisfies a bound of the form
\begin{align}\label{eq:kappa2}
g(sx+(1-s)y) \leq sg(x) + (1-s)g(y) - \frac{\kzwei}{2}s(1-s)\norm{x-y}{X}^2
\end{align}
for all $x,y \in X$ and all $s \in [0,1]$ for some $\kzwei \in \RR$. For $\kzwei > 0$ estimate \eqref{eq:kappa2} represents $\kzwei$-strong convexity of $g$. It is known that $\kzwei$-strong convexity of $g$ implies that $g$ is bounded from below, its level-sets $L_\alpha g$ bounded for all $\alpha \in \RR$ and their diameter shrinks to $0$, if $\alpha \to \inf_{x\in X} g$. 
In the case of $\kzwei < 0$, $g$ is allowed to be non-convex in a limited way. 

The theory behind Proximal Newton methods and the respective convergence properties evolves around the convexity estimates stated in \eqref{eq:kappa1} and \eqref{eq:kappa2}. We will assign particular importance to the interplay of the convexity properties of $f$ and $g$, i.e., the sum $\ksum$ will continue to play an important part over the course of the present treatise. 

Let us now shortly outline the structure of our work: In Section~\ref{sec:local} we will consider undamped update steps computed as the solution of an adequately formulated subproblem. These can also be represented using (scaled) proximal mappings the definition and key properties of which we shortly address. Afterwards, local superlinear convergence of the Proximal Newton method is shown. In Section~\ref{sec:global} we present a modification of the aforementioned subproblem in order to damp update steps and globalize the Proximal Newton method. This enables the proof of optimality of all limit points of the sequence of iterates generated by our method. Section~\ref{sec:soss} concerns the introduction of second order semi-smoothness for $f$ and showcases how it helps to verify the admissibility of both full and damped update steps sufficiently close to optimal solutions in Section~\ref{sec:transition}. This in turn enables local fast convergence of our globalized method. In Section~\ref{sec:num} the performance of our algorithm is substantiated by numerical results.

As a start, we want to introduce the definition of undamped update steps and investigate the behavior of the ensuing Proximal Newton method close to optimal solutions of problem \eqref{eq:prob}.

\section{General Dual Proximal Mappings} \label{sec:prox}
We compute a full step for the Proximal Newton method at a current iterate $x \in X$ by solving the subproblem
\begin{align} \label{eq:fullstep}
\Delta x \coloneqq \argmin_{\delta x \in X} f'(x)\delta x + \frac 12 H_x (\delta x, \delta x) + g(x + \delta x) - g(x).
\end{align}
In this section $H_x$ denotes a general bilinear form, as introduced above. If a minimizer exists, we determine the next iterate via $x_+ \coloneqq x + \Delta x$. We will consider this update scheme and investigate its convergence properties close to optimal solutions, and in particular fast local convergence if $H_x$ is adequately chosen as a so-called Newton derivative from $\partial_N f'(x)$, also known as the generalized differential $\partial^* f'(x)$ in the sense of Chapter 3.2 in \cite{Ulbrich2011}.

\begin{proposition}
	If $\keins+\kzwei>0$, then \eqref{eq:fullstep} admits a unique solution. 
\end{proposition}
\begin{proof}
	By assumption, the functional to be minimized is lower semi-continuous, and $\keins+\kzwei>0$ implies that it is strictly convex as well as radially unbounded. Since $X$ is a Hilbert space a minimizer exists and is unique. \qed
\end{proof}
\begin{remark}
	Let us shortly elaborate on both constants $\keins$ and $\kzwei$ as well as the assumption $\ksum > 0$. While $\kzwei$ is a global convexity constant for $g$, $\keins$ is a purely local quantity which differs from iterate to iterate together with the corresponding second order bilinear form $H_{x_k}$. This has two immediate consequences: On the one hand, ellipticity of the second order bilinear forms can locally compensate for non-convexity of $g$ and on the other hand (global) convexity of $g$ enables us to locally use non-elliptic $H_x$ even close to optimal solutions of our minimization problem. Comparing these convexity assumptions to similar works on the topic, we recognize that the authors in both \cite{Lee2014} and \cite{Kanzow2020} require ellipticity of their $\nabla^2 f (x_*)$ in addition to convexity of $g$. In contrast, our $(\keins,\kzwei)$-formalism from above suitably quantifies the contribution to convexity of both $f$ and $g$. 
\end{remark}
For the following discussion we keep the assumption $\keins+\kzwei>0$. To introduce an adequate definition of a proximal mapping in Hilbert space we reformulate \eqref{eq:fullstep} directly for the updated iterate $x_+$ via
\begin{align} \label{eq:preproxmin}
x_+ = \argmin_{y \in X} f'(x)(y-x) + \frac 12 H_x(y-x,y-x) + g(y) - g(x) \, .
\end{align}
In the literature existence of a continuous inverse $H_x^{-1} : X^* \to X$ is frequently assumed, giving rise to a mapping $H_x^{-1}f': X \rightarrow X$. Then \eqref{eq:preproxmin}  can be rearranged to
\begin{align} \label{eq:updateLee}
x_+ = \argmin_{y \in X} g(y) + \frac 12 H_x(y- \big(x - H_x^{-1}f'(x)\big))^2 \, .
\end{align}
In \cite{Lee2014}, this form of the updated iterate is considered and the notion of a proximal mapping is introduced by 
\begin{align*}
\mathrm{prox}_g^H (x) \coloneqq \argmin_{y \in \RR^n} \, g(y) + \frac{1}{2} (y-x)^T H (y-x) = \argmin_{y \in \RR^n} \, g(y) + \frac{1}{2} \norm{y-x}{H}^2
\end{align*}
such that there \eqref{eq:updateLee} takes the form $x_+ =\mathrm{prox}_g^{H_x} \big(x - H_x^{-1}f'(x)\big)$.

However, in this work we want to follow a different, more direct approach towards proximal mappings which allows us to use the structure of the dual space $X^*$ more accurately and dispense with an invertibility assumption on $H_x$. In \cite{TranDinh2015} (scaled) proximal mappings are introduced for $X=\RR^n$ according to
\begin{align*}
\mathcal{P}_g^H : \RR^n \to \RR^n \, , \, \mathcal{P}_g^H (x) \coloneqq \argmin_{y \in \RR^n} \, g(y) + \frac{1}{2} y^T H y - x^T y \, .
\end{align*}
Observing that $x^T$ represents a dual element in $\RR^n$ here, we generalize this notion to the setting of Hilbert spaces and consider
\begin{align} \label{eq:proxmin}
\mathcal{P}_g^H : X^* \to X \, , \, \mathcal{P}_g^H (\varphi) \coloneqq \argmin_{y \in X} \, g(y) + \frac{1}{2} H(y,y) - \varphi(y),
\end{align}
obtaining a mapping from the dual space back to the primal space. 

With this definition in mind, \eqref{eq:preproxmin} can directly be rewritten as
\begin{align} \label{eq:proxupdate}
x_+ = \argmin_{y \in X} \, g(y) + \frac{1}{2} H_x(y)^2 - \big( H_x (x) - f'(x) \big)(y) = \mathcal{P}_g^{H_x} (H_x (x) - f'(x)) \,.
\end{align}
Our notion allows us to dispense with the use of the inverse $H_x^{-1}$, which would require in addition $\kappa_1>0$. 
We will refer to \eqref{eq:proxmin} as the direct or dual formulation of scaled proximal mappings.

First order conditions for the minimization problem posed in \eqref{eq:proxupdate} yield the equation 
\begin{align*}
\eta + H_x(x_+ - x) + f'(x) = 0
\end{align*}
in the dual space $X^*$ for some (Frech\'et-)subderivative $\eta \in \partial_F g (x_+)$ (if $g$ is convex, $\partial_F g$ coincides with the convex subdifferential $\partial g$, cf. \cite{Kruger2003}). As we rearrange this identity, one could formally write:
\begin{align*}
x_+ = \left(H_x + \partial_F g \right)^{-1}\left( H_x - f' \right) x \, .
\end{align*}
If $H_x$ is additionally invertible, this is equivalent to
\begin{align*}
x_+ = \left(\Id + H_x^{-1} \partial_F g \right)^{-1}\left( \Id - H_x^{-1}f' \right) x
\end{align*}
which once again substantiates the interpretation of proximal-type methods as forward-backward splitting algorithms. Note that in particular the subdifferential of $g$ is evaluated at the updated point $x_+$.

We can shift convexity properties of the respective parts of the composite objective functional by inserting adequate bilinear form terms. However, this procedure does not affect the sequence of iterates generated by the update formula from above:
\begin{lemma} \label{lem:Fschlange}
	Let $q:X \to \RR$ be a continuous quadratic function and denote its second derivative (which is independent of $x$) by $Q \coloneqq  q''(x) : X \to X^*$.  
	Consider the modified (but obviously equivalent) minimization problem
	\begin{align} \label{eq:modprob}
	\min_{x \in X} \tilde F(x) &\coloneqq \tilde f (x) + \tilde g (x)\\
	\tilde f(x) &\coloneqq  f(x) - q(x), \qquad \tilde g (x) \coloneqq  g(x) + q(x).
	\end{align}
	Then, the update steps computed via \eqref{eq:proxupdate} are identical for both problems \eqref{eq:prob} and \eqref{eq:modprob} if we choose $\tilde H_x = H_x - Q$ as the corresponding bilinear form.
\end{lemma}
\begin{remark}
	If we choose $q(x) \coloneqq \frac \kappa 2 \norm{x}{X}^2$ for some $\kappa \in \RR$, the modified quantities $\tilde H_x$ and $\tilde g$ suffice estimates \eqref{eq:kappa1} and \eqref{eq:kappa2} for $\tilde \kappa_1 = \kappa_1 - \kappa$ and $\tilde \kappa_2 = \kappa_2 + \kappa$. In particular, $\kappa_1 + \kappa_2 = \tilde \kappa_1 + \tilde \kappa_2$ remains unchanged and $\tilde g$ is $(\kappa + \kappa_2)$-strongly convex for $\kappa > -\kappa_2$.
\end{remark}
\begin{proof}
	The only claim which is not apparent is the identity of update steps. To this end, we consider the fundamental definition of the update step for problem \eqref{eq:modprob} at some $x \in X$ given by
	\begin{align*}
	\Delta \tilde x = \argmin_{\delta x \in X} \tilde f'(x)\delta x + \frac 12 \tilde H_x (\delta x)^2 + \tilde g(x + \delta x) - \tilde g(x)
	\end{align*}
	and consequently for $q(y) = \frac 12 Q (y)^2 + ly + c$ and $c \in \RR$ constant
	\begin{align*}
	\tilde x_+ &= \argmin_{y \in X} \big(f'(x) - q'(x)\big)(y-x) + \frac 12 \big(H_x - q''(x)\big)(y-x)^2 + g(y) + q(y) \\
	&= \argmin_{y \in X} \big(f'(x) - (Qx + l) \big)(y-x) + \frac 12 \big(H_x - Q\big)(y-x)^2 + g(y) + \frac 12 Q (y)^2 + ly \\
	&= \argmin_{y \in X} g(y) + \frac 12 H_x (y)^2 - \big( (H_x - Qx) - (f'(x)-Qx) \big)y \\
	&= \mathcal{P}_g^{H_x} (H_x (x) - f'(x)) = x_+
	\end{align*}
	which directly shows the asserted identity of update steps. \qed
\end{proof}
\begin{remark}
	If the bilinear form for update step computation is chosen as $H_x \in \partial_N f'(x)$ and thereby as $\tilde H_x \in \partial_N \tilde f'(x)$ in the modified case, we have $\tilde H_x = H_x - Q$, automatically.
\end{remark}

\section{Regularity and Fast Local Convergence}\label{sec:local}

The representation of the updated iterate as the image of a scaled proximal mapping in \eqref{eq:proxupdate} will turn out to be very useful in what follows which is why we dedicate the next two propositions to the properties of scaled proximal mappings in our scenario. The first proposition generalizes the assertions of the so called second prox theorem, cf. e.g. \cite{Beck2017}, to our notion of proximal mappings.
\begin{proposition}\label{prop:scndprox}
	Let $H$ and $g$ satisfy the assumptions \eqref{eq:kappa1} and \eqref{eq:kappa2} with $\ksum > 0$. Then for any $\varphi \in X^*$ the image of the corresponding proximal mapping $u \coloneqq \p_g^H(\varphi)$ satisfies the estimate
	\begin{align*}
	\big[ \varphi - H(u) \big](\xi - u) \leq g(\xi) - g(u) - \frac{\kzwei}{2}\norm{\xi - u}{X}^2
	\end{align*}
	for all $\xi \in X$.
\end{proposition}
\begin{proof}
	The proof of the estimate above is an easy consequence of the characterization of the convex subdifferential of $g_H \coloneqq g +\frac 12 H(\cdot, \cdot)$ and \eqref{eq:kappa2}. First order conditions of the minimization problem in \eqref{eq:proxmin} yield
	\begin{align*}
	\varphi \in \partial \big(g + \frac 12 H(\cdot,\cdot) \big)(u) = \partial g_H (u)
	\end{align*}
	where $\partial$ denotes the convex subdifferential since in particular $g_H$ is convex due to the positivity of the sum $\ksum$. This inclusion directly implies the estimate
	\begin{align*}
	\varphi (y-u) + g(u) + \frac 12 H(u,u) \leq g(y) + \frac 12 H(y,y)
	\end{align*}
	for arbitrary $y \in X$ which is equivalent to
	\begin{align*}
	\big[ \varphi - \frac 12 H(y+u) \big](y-u) \leq g(y) - g(u) \, .
	\end{align*}
	As pointed out before, now we want to take advantage of the convexity assumptions on $g$ according to \eqref{eq:kappa2}. To this end, we insert $y = y(s) \coloneqq s\xi + (1-s)u$ above for $s \in ]0,1]$ and use \eqref{eq:kappa2} on the right-hand side. This yields
	\begin{align*}
	s\big[ \varphi - H(u) - \frac s2 H(\xi - u) \big](\xi - u) \leq s\big[g(\xi) - g(u) - \frac{\kzwei}{2}(1-s)\norm{\xi - u}{X}^2\big]
	\end{align*}
	where we now divide by $s \neq 0$ and subsequently evaluate the limit of $s$ to zero. This procedure provides us with the asserted estimate for $\xi$, $\varphi$ and $u$ as specified above. \qed
\end{proof}
The inequality from Proposition~\ref{prop:scndprox} can be used in order to prove several useful continuity results for general scaled proximal mappings in Hilbert spaces. However, for our purposes it suffices to assert and verify the following result, which generalizes non-expansivity of proximal mappings in Euclidean space to our setting. It plays a similar role as boundedness of the inverse of the derivative in Newton's method. 
\begin{corollary}[Regularity of the Prox-Mapping] \label{cor:proxcontinuity}
	Let $H$ and $g$ satisfy the assumptions \eqref{eq:kappa1} and \eqref{eq:kappa2} with $\ksum > 0$. Then, for all $\varphi_1,\varphi_2 \in X^*$ the following Lipschitz-estimate holds:
	\begin{align*}
	\norm{\p_g^H(\varphi_1) - \p_g^H(\varphi_2)}{X} \leq \frac{1}{\ksum} \norm{\varphi_1 - \varphi_2}{X^*}
	\end{align*} 
\end{corollary}
\begin{proof}
	Let us choose $H$ and $\varphi_1,\varphi_2$ as stated above. According to Proposition~\ref{prop:scndprox}, the first order conditions for the respective minimization problems yield the inequalities 
	\begin{align}
	(\varphi_1 - H(u_1))(u_2- u_1)\leq g(u_2) - g(u_1) - \frac{\kzwei}{2}\norm{u_2 - u_1}{X}^2 \label{eq:veinsfertig} \\
	(\varphi_2 - H(u_2))(u_1- u_2)\leq g(u_1) - g(u_2) - \frac{\kzwei}{2}\norm{u_1 - u_2}{X}^2  \label{eq:vzweifertig}
	\end{align}
	since we can choose $\xi \coloneqq u_2$ or $\xi \coloneqq u_1$ respectively. Now, we add \eqref{eq:veinsfertig} and \eqref{eq:vzweifertig} which yields
	\begin{align*}
	(\varphi_2 - \varphi_1 + H(u_1-u_2))(u_1-u_2) \leq - \kzwei \norm{u_1 - u_2}{X}^2 \, .
	\end{align*}
	As we rearrange this inequality we obtain
	\begin{align*}
	H(u_1-u_2)^2 + \kzwei \norm{u_1 - u_2}{X}^2 \leq (\varphi_1 - \varphi_2)(u_1-u_2) \leq \norm{\varphi_1 - \varphi_2}{X^*}\norm{u_1-u_2}{X}
	\end{align*}
	and eventually assumption \eqref{eq:kappa1} on $H$ yields the assertion of the proposition. \qed
\end{proof}
Even though the above continuity result for proximal mappings will turn out to be an important tool for the proof of local acceleration of the Proximal Newton method, we still have to deduce some crucial properties of the full update step $\Delta x$. These will help us to characterize optimal solutions of \eqref{eq:prob} as fixed points of the method and then verify local acceleration afterwards.
\begin{lemma} \label{lem:descentdirfull}
	The undamped update steps computed via \eqref{eq:fullstep} are descent directions of the composite objective functional, i.e., the following estimate holds:
	\begin{align*}
	F(x+s\Delta x) \leq F(x) - s(\ksum)\norm{\Delta x}{X}^2 + O(s^2) \, .
	\end{align*}
\end{lemma}
\begin{proof}
	Since $f$ is assumed to be continuously differentiable and $g$ suffices the estimate \eqref{eq:kappa2}, we can deduce the following bound on the composite objective functional:
	\begin{align} \label{eq:descentdirtaylor}
	\begin{split}
		F(x+s\Delta x) &\leq f(x) + sf'(x)\Delta x + O(s^2) \\
		&\quad + sg(x+\Delta x) + (1-s)g(x) - \frac{\kzwei}{2}s(1-s)\norm{\Delta x}{X}^2  \\
		&\leq F(x) + s(f'(x)\Delta x + g(x+ \Delta x) - g(x) - \frac{\kzwei}{2}\norm{\Delta x}{X}^2) + O(s^2)
	\end{split}
	\end{align}
	Let us now deduce an estimate for the term in brackets on the right-hand side of \eqref{eq:descentdirtaylor}. To this end, we remember the proximal mapping representation of updated iterates in \eqref{eq:proxupdate} and consider the corresponding estimate from Proposition~\ref{prop:scndprox} for $\xi \coloneqq x$ which is given by
	\begin{align*}
	\big[ H_x(x) - f'(x) - H_x(x_+) \big](x-x_+) \leq g(x) - g(x_+) - \frac{\kzwei}{2}\norm{x_+ - x}{X}^2
	\end{align*}
	or equivalently 
	\begin{align*}
	f'(x)\Delta x + g(x+\Delta x) - g(x) - \frac{\kzwei}{2}\norm{\Delta x}{X}^2 &\leq - H_x(\Delta x)^2 - \kzwei \norm{\Delta x}{X}^2 \\ &\leq - (\ksum)\norm{\Delta x}{X}^2
	\end{align*}
	which we insert into \eqref{eq:descentdirtaylor} and directly obtain the asserted inequality. Note that over the course of this section we assume the positivity of the sum $\keins+\kzwei$ which indeed implies from above that $\Delta x$ is a descent direction. \qed
\end{proof}
As mentioned beforehand, this directly enables a more insightful characterization of optimal solutions of the composite minimization problem.
\begin{proposition} \label{prop:optsolcharfull}
	Consider $f$ continuously differentiable with Lipschitz derivative as well as $H \in \Lin$ which satisfies \eqref{eq:kappa1} with $\ksum > 0$ and $\kzwei$ from \eqref{eq:kappa2} for $g$. Then, the search direction $\Delta x_*$ according to \eqref{eq:fullstep} is zero at every local minimizer $x_* \in X$ of problem \eqref{eq:prob}. In particular, we obtain the fixed point equation 
	\begin{align*}
	x_* = \p^{H}_g \big(H (x_*) - f'(x_*) \big) \, .
	\end{align*}
\end{proposition}
\begin{proof}
	If $x_*$ is a local minimizer, $F(x_*+s\Delta x) \geq F(x_*)$ for sufficiently small $s>0$. By Lemma~\ref{lem:descentdirfull} this implies $\Delta x=0$.  \qed
\end{proof}

Having in mind these properties of update steps and optimal solutions in addition to the continuity result for scaled proximal mappings from Proposition~\ref{cor:proxcontinuity}, we can now prove the local acceleration result for our Proximal Newton method with undamped steps near optimal solutions.

For the following we require $f'$ to be semi-smooth near an optimal solution $x_*$ of our problem \eqref{eq:prob} with respect to $H_x$, i.e., the following approximation property holds:
\begin{align} \label{eq:semismooth}
\norm{f'(x_*) - f'(x) - H_{x}(x_* - x)}{X^*} = o\big( \norm{x-x_*}{X} \big) \, .
\end{align}
As pointed out before, adequate definitions of $H_x$ can be given via the Newton derivative $H_x \in \partial_N f'(x)$ for Lipschitz-continuous operators in finite dimension as well as for corresponding superposition operators, cf. Chapter 3.2 in \cite{Ulbrich2011}.

\begin{theorem}[Fast Local Convergence] \label{thm:localacc}
	Suppose that $x_* \in X$ is an optimal solution of problem \eqref{eq:prob}. Consider two consecutive iterates $x, x_+ \in X$ which have been generated by the update scheme from above and are close to $x^*$. Furthermore, suppose that \eqref{eq:semismooth} holds for $H_x$ in addition to the assumptions from the introductory section with $\ksum > 0$. Then we obtain:
	\begin{align*}
		\norm{x_+ - x_*}{X} = o\big( \norm{x-x_*}{X} \big) \quad \text{in the limit of} \quad x \to x^* .
	\end{align*}
\end{theorem}
\begin{proof}
	Consider the proximal mapping representations deduced above for both the updated iterate $x_+$ in \eqref{eq:proxupdate} and for the optimal solution $x_*$ in Proposition~\ref{prop:optsolcharfull} via 
	\begin{align*}
	x_+ = x + \Delta x = \p^{H_x}_g \big(H_x (x) - f'(x) \big) \quad \text{and} \quad x_* = \p^{H_x}_g \big(H_x (x_*) - f'(x_*) \big) \, .
	\end{align*}
	Next, we directly take advantage of these identities together with the continuity result for scaled proximal mappings from Proposition~\ref{cor:proxcontinuity} in order to deduce the estimate
	\begin{align*}
	\norm{x_+ - x_*}{X} &= \norm{\p^{H_x}_g \big(H_x (x) - f'(x) \big) - \p^{H_x}_g \big(H_x (x_*) - f'(x_*) \big)}{X} \\
	&\leq \frac{1}{\ksum} \norm{H_x (x) - f'(x) - (H_x (x_*) - f'(x_*))}{X^*} \\
	&= \frac{1}{\ksum} \norm{H_x(x - x_*) - (f'(x) - f'(x_*))}{X^*} \\
	&= o\big( \norm{x-x_*}{X} \big)
	\end{align*}
	where in the last step also the semi-smoothness of $f'$ played a crucial role. This directly verifies the asserted local acceleration result. \qed
\end{proof}

In particular, this implies local superlinear convergence of our Proximal Newton method if we can additionally verify global convergence to an optimal solution. Note that even for the local acceleration result, ellipticity of $H_x \in \partial_N f'(x)$ does not necessarily have to be demanded. Also here, all that matters is strong convexity of the composite functional. This might be surprising since what actually accelerates the method is the second order information on the (possibly non-convex) but differentiable part $f$ with semi-smooth derivative $f'$. As a consequence, this means that the (strong) convexity of $g$ can not only contribute to the well-definedness of update steps as solutions of \eqref{eq:fullstep} but also to the local acceleration of our algorithm.

The main reason for this generalization of the local acceleration result is our slightly generalized notion of proximal mappings. In particular, we did not deduce (firm) non-expansivity in the scaled norm as for example in \cite{Lee2014} but also there took advantage of the strong convexity of the composite objective functional in the form of assumptions \eqref{eq:kappa1} and \eqref{eq:kappa2} with $\ksum > 0$.

Note that for the above results to hold it was crucial that the current iterate $x$ is already close to an optimal solution of problem \eqref{eq:prob} which is why over the course of the next section we want to address one possibility to globalize our Proximal Newton method. We will see that eventually we will be in the position to use undamped update steps for the computation of iterates and thereby benefit from the local acceleration result in Theorem~\ref{thm:localacc}.

\section{Globalization via an additional norm term} \label{sec:global}
Let us consider the following modification of \eqref{eq:fullstep} and define the damped update step at a current iterate $x$ as a minimizer of the following modified model functional:
\[
\lambda_\omega(\delta x) \coloneqq  f'(x)\delta x + \frac 12 H_x (\delta x, \delta x) + \frac{\omega}{2}\norm{\delta x}{X}^2 + g(x+\delta x) - g(x) \, .
\]
As a consequence, we define
\begin{align} \label{eq:dampedstep}
\Delta x (\omega) \coloneqq \argmin_{\delta x \in X} \lambda_\omega(\delta x) \, .
\end{align}
Here $\omega>0$ is an algorithmic parameter that can be used to achieve global convergence.
Setting $\tilde H \coloneqq  H_x+\omega \mathcal R$ with the Riesz-isomorphism $\mathcal R:X \to X^*$ we observe that \eqref{eq:dampedstep} is of the form~\eqref{eq:fullstep} with $\tilde \kappa_1 = \keins+\omega$, so that the existence and regularity results of the previous sections apply.  

The updated iterate then takes the form $x_+ (\omega) \coloneqq x + \Delta x (\omega)$. Apparently, the update step in \eqref{eq:dampedstep} is well defined if $\omega + \ksum > 0$. Consequently, for what follows, we only consider $\omega > -(\ksum)$ in order to guarantee unique solvability of the update step subproblem. The full update steps from \eqref{eq:fullstep} are here damped along a curve in $X$ which is parametrized by the regularization parameter $\omega \in ]-(\ksum),\infty[$. 

However, note that here the Hilbert space structure of $X$ is also important for the strong convexity of functions of the form $g +\frac{\omega}{2}\norm{\cdot}{X}^2$ with $g$ as in \eqref{eq:kappa2} for arbitrary $\kappa_2 \in \RR$. In a general Banach space setting, we can not assume additional norm terms to compensate disadvantageous convexity assumptions, cf. \cite{Beck2017}, Remark 5.18].

Let us now take a look at how we can rearrange the subproblem for finding an updated iterate by using the scalar product $\scpr{\cdot}{\cdot}$ as well as the Riesz-Isomorphism $\mathcal R$:
\begin{align}
\begin{split}
x_+ (\omega) &= \argmin_{y \in X} f'(x)(y-x) + \frac 12 H_x (y-x, y-x) + g(y) - g(x) + \frac{\omega}{2}\norm{y-x}{X}^2  \\
&= \argmin_{y \in X} g(y) + f'(x) y + \frac 12 H_x (y)^2 - H_x(x,y) + \frac{\omega}{2} \norm{y}{X}^2 - \omega \scpr{x}{y}  \\
&= \argmin_{y \in X} g(y) + \frac 12 \big(H_x + \omega \mathcal{R} \big) (y)^2 - \big( H_x (x) + \omega \mathcal{R} x - f'(x) \big)y  \\
&= \p_{g}^{H_x + \omega \mathcal{R}} \big( (H_x + \omega \mathcal{R}) x - f'(x) \big) \, .
\end{split}
\label{eq:dampprox}
\end{align}
Note that $H_x + \omega \mathcal{R}: X \times X \to \RR$ satisfies \eqref{eq:kappa1} with constant $(\kappa_1 + \omega)$ such that the combination of $g$ and $H_x + \omega \mathcal{R}$ still suffices the requirements for the results from Proposition~\ref{prop:scndprox} for all $\omega > -(\ksum)$. Additionally, the results of Lemma~\ref{lem:Fschlange} apparently also hold in the globalized case.

The formulation of updated iterates via the above scaled proximal mapping enables us to establish some helpful properties of the damped update steps $\Delta x(\omega)$.
\begin{proposition} \label{prop:dampedstepprop}
	Under the assumptions \eqref{eq:kappa1} for $H_x$ and \eqref{eq:kappa2} for $g$ the inequality
	\begin{align*}
	f'(x)\Delta x(\omega) + g(x+\Delta x(\omega)) - g(x) \leq -\big( \frac{\kzwei}{2} + \omega \big) \norm{\Delta x(\omega)}{X}^2 - H_x(\Delta x(\omega))^2
	\end{align*}
	holds for the update step $\Delta x(\omega)$ as defined in \eqref{eq:dampedstep} and arbitrary $-(\ksum) < \omega < \infty$.
\end{proposition}
\begin{proof}
	The proof here follows along the same lines as the derivation of the auxiliary estimate for the bracket term in the proof of Lemma~\ref{lem:descentdirfull}. Due to the structure of the update formula in \eqref{eq:dampprox} we can take advantage of the estimate from Proposition~\ref{prop:scndprox} with $\varphi = (H_x + \omega \mathcal{R}) x - f'(x)$, $H = H_x + \omega \mathcal{R}$ and $\xi = x$ which yields $u = \p_{g}^H (x) = x_+$ and thereby 
	\begin{align*}
	H_x (\Delta x (\omega))^2 + \omega \norm{\Delta x (\omega)}{X}^2 + f'(x)\Delta x (\omega) \leq g(x) - g(x_+(\omega)) - \frac{\kappa_2}{2} \norm{\Delta x (\omega)}{X}^2 \, .
	\end{align*} 
	This inequality is equivalent to the asserted estimate. \qed
\end{proof}
With the above estimate for damped update steps at hand, let us now formulate a criterion for sufficient decrease which will help us to verify a global convergence result of our Proximal Newton method. We call a value of the regularization parameter $\omega > -(\ksum)$ admissible for sufficient decrease if the inequality
\begin{align}\label{eq:lambda}
F(x_+(\omega)) \leq F(x) + \gamma\lambda_\omega(\Delta x(\omega))
\end{align}
for some prescribed $\gamma \in ]0,1[$ is satisfied. We may interpret $\lambda_\omega(\Delta x(\omega))$ as a predicted decrease and rewrite the condition \eqref{eq:lambda} as follows:
\[
\frac{F(x_+(\omega)) -F(x)}{\lambda_\omega (\Delta x(\omega))} \ge \gamma \, .
\]
This is the classical ratio of actual decrease and predicted decrease which is often used for trust-region algorithms. Before now trying to verify that the descent criterion in \eqref{eq:lambda} is fulfilled for sufficiently large values of $\omega$, we note that the assertion in Proposition~\ref{prop:dampedstepprop} implies the insightful estimate
\begin{equation} \label{eq:lambdanorm}
\begin{split}
\lambda_\omega(\Delta x(\omega)) &\leq -\big( \frac{\kzwei}{2} + \omega \big) \norm{\Delta x(\omega)}{X}^2 - \frac 12 H_x\big(\Delta x(\omega)\big)^2+ \frac \omega2 \|\Delta x(\omega)\|_X^2\\
&\leq - \frac12\big( \omega +  \ksum\big)\norm{\Delta x(\omega)}{X}^2
\end{split}
\end{equation}
which yields that once the criterion is satisfied, update steps unequal to zero provide real descent in the composite objective function $F$ according to
\begin{equation}\label{eq:decrease}
F(x_+(\omega))-F(x) \le -\frac{\gamma}{2} \big( \omega +  \ksum\big)\norm{\Delta x(\omega)}{X}^2 \, . 
\end{equation}
Let us now take a look at the existence of sufficiently large values of the regularization parameter $\omega$. Here, the Lipschitz-continuity of $f'$ comes into play for the first time.
\begin{lemma} \label{lem:sufflarge}
	For $f$, $H_x$ and $g$ as above the criterion for sufficient descent introduced via \eqref{eq:lambda} is satisfied for $\gamma \in ]0,1[$ if $\omega$ satisfies 
	\[
	\omega \ge \frac{L_f-\kappa_1}{1-\gamma}-(\ksum) \, .
	\]
\end{lemma}
\begin{proof}
	By our lower bound on $\omega$ and \eqref{eq:lambdanorm} we obtain:
	\[
	\frac{L_f - \keins}{2}\|\Delta x(\omega)\|^2_X \le \frac{1-\gamma}{2}(\omega+\ksum)\|\Delta x(\omega)\|^2_X \le -(1-\gamma)\lambda_\omega(\Delta x(\omega)) \, . 
	\]
	The Lipschitz-continuity of $f'$ directly yields the estimate
	\begin{align*}
	f(x_+(\omega)) = f(x+\Delta x(\omega)) \leq f(x) + f'(x)\Delta x(\omega) + \frac{L_f}{2}\norm{\Delta x(\omega)}{X}^2
	\end{align*}
	from where we immediately obtain an estimate for the descent in the composite objective functional via
	\begin{align*}
	F(x_+(\omega)) - F(x) &\leq f'(x)\Delta x(\omega) + \frac{L_f}{2}\norm{\Delta x(\omega)}{X}^2 + g(x_+(\omega)) - g(x) \\
	&\leq \lambda_\omega (\Delta x(\omega)) + \frac{L_f - \keins}{2}\norm{\Delta x(\omega)}{X}^2\\
	&\leq \lambda_\omega (\Delta x(\omega))-(1-\gamma)\lambda_\omega(\Delta x(\omega))=\gamma\lambda_\omega(\Delta x(\omega)) \, .
	\end{align*}
	This estimate is equivalent to \eqref{eq:lambda} and thereby concludes the proof of the assertion. \qed
\end{proof}

Additionally, for global convergence, it turns out that we have to guarantee that
\[
\lambda_{\omega_k}(\Delta x(\omega_k)) \to 0 \mbox{ implies } \|\Delta x(\omega_k)\|_X \to 0. 
\]
A simple way to achieve this is to impose the following restriction:
\begin{align} \label{eq:lambdaMbound}
\|\Delta x(\omega_k)\|_X^2 \le -\overline M \lambda_{\omega_k}(\Delta x(\omega_k))
\end{align}
for some prescribed upper bound $\overline M$. Due to \eqref{eq:lambdanorm} this can be achieved for a sufficiently large choice of $\omega_k$. All in all, this results in the following algorithm: \\
\begin{algorithm}[H] \label{alg:proxnewton}
	\KwData{Starting point $x_0 \in X$, sufficient decrease parameter $\gamma \in ]0,1[$, $\varepsilon > 0$ for stopping criterion}
	Initialization: $k=0$\;
	\While{$(1+\omega_k)\norm{\Delta x_k(\omega_k)}{X} \geq \varepsilon$}{
		Compute a trial step $\Delta x_k (\omega_k)$ according to \eqref{eq:dampedstep}\;
		\eIf{bound \eqref{eq:lambdaMbound} and sufficient descent criterion \eqref{eq:lambda} are satisfied}{
			Update current iterate to $x_{k+1} \leftarrow x_k + \Delta x_k(\omega_k)$\;
			Decrease $\omega_{k}$ to some $\omega_{k+1} < \omega_k$ for next iteration\;
			Update $k \leftarrow k+1$ \;
		}{
			Increase $\omega_k$ appropriately\;
		}
	}
	\caption{Second order semi-smooth Proximal Newton algorithm damped according to \eqref{eq:dampedstep}}
\end{algorithm}

Now that we have formulated the algorithm and can be sure that we can always damp update steps sufficiently such that they yield descent according to \eqref{eq:lambda}, we will verify the stationarity of limit points of the sequence of iterates generated by Algorithm~\ref{alg:proxnewton}. To this end, we will first prove that the norm of the corresponding update steps converges to zero along the sequence of iterates. 
\begin{lemma} \label{lem:dxtozero}
	Let $(x_k) \subset X$ be the sequence generated by the Proximal Newton method globalized via \eqref{eq:dampedstep} for admissible values of the regularization parameter $\omega_k$ starting at any $x_0 \in \mathrm{dom}g$. Then either $F(x_k) \to -\infty$ or $\lambda_{\omega_k}(\Delta x_k(\omega_k))$ and $\norm{\Delta x_k(\omega_k)}{X}$ converge to zero for $k \to \infty$.
\end{lemma}
\begin{proof}
	By \eqref{eq:decrease} the sequence $F(x_k)$ is monotonically decreasing. Thus, either $F(x_k) \to -\infty$ or $F(x_k) \to \underline F$ for some $\underline F\in \RR$ and thus in particular $F(x_{k})-F(x_{k+1})\to 0$. Since $\gamma > 0$, also $\lambda_{\omega_k}(\Delta x(\omega))\to 0$. Since, by assumption, $\omega_k +\ksum > 0$ this implies $\norm{\Delta x_k(\omega_k)}{X} \to 0$.\qed
\end{proof}
If we take a look at the optimality conditions for the step computation in \eqref{eq:dampedstep} at $x_+(\omega)$, we obtain
\begin{align*}
(H_x+ \omega \mathcal{R}) x - f'(x) \in \partial_F g_\omega^{H_x} (x_+(\omega))
\end{align*}
with the Frech\'et-subdifferential of $g_\omega^{H_x}:X\to \RR, y \mapsto g(y) + \frac 12 H_x(y)^2 + \frac{\omega}{2}\norm{y}{X}^2$ on the right-hand side. This directly yields the existence of some $\eta \in \partial_F g(x_+(\omega))$ such that
\begin{align} \label{eq:subdiffdist}
\eta + f'(x_+(\omega)) = r_x(\Delta x(\omega)) \quad \text{with} \quad r_x(v)\coloneqq f'(x+v)-f'(x)- \big(H_x + \omega \mathcal{R}\big)v \, .
\end{align}
This implies the estimate: 
\begin{align*}
\mathrm{dist}(\partial_F F (x_k),0) = \mathrm{dist}(f'(x_k) + \partial_F g(x_k),0) \leq \norm{r_{x_k}(\Delta x_k(\omega_k))}{X^*}
\end{align*} 
Thus, by Lemma~\ref{lem:dxtozero} and
\begin{align*}
\|r_{x_k}(\Delta x_k(\omega_k))\|_{X^*} \leq \big(L_f + \norm{H_{x_k}}{\Lin} + \omega_k\big)\norm{\Delta x_k (\omega_k)}{X}
\end{align*}
we obtain
\[
\mathrm{dist}(\partial_F F (x_k),0) \to 0
\]
as long as $L_f< \infty$ exists, $\norm{H_{x_k}}{\Lin} \le M$ is bounded, and $\omega_k$ is bounded. The latter can be guaranteed via Lemma~\ref{lem:sufflarge} if the ``appropriate increase'' of $\omega_k$ is done by no more than a fixed factor $\rho >1$.  

\begin{remark}
	With some additional technical effort, the assumption on Lipschitz-continuity of $f'$ could be relaxed to a uniform continuity assumption. 
\end{remark}

Observe that we can indeed interpret $\norm{\Delta x_k(\omega_k)}{X} \leq \varepsilon$ as a condition for the optimality of our the subsequent iterate up to some prescribed accuracy. However, small step norms $\norm{\Delta x_k(\omega_k)}{X}$ can also occur due to very large values of the damping parameter $\omega_k$ as a consequence of which the algorithm would stop even though the sequence of iterates is not even close to an optimal solution of the problem. In order to rule out this inconvenient case, we consider the scaled version $(1+\omega_k)\norm{\Delta x_k(\omega_k)}{X}$ as the stopping criterion in Algorithm~\ref{alg:proxnewton}.

Now we are in the position to discuss subsequential convergence of our algorithm to a stationary point. In the following, we will assume throughout that $F(x_k)$ is bounded from below.  We start with the case of convergence in norm:
\begin{theorem} \label{thm:strongstat}
	Under the assumptions explained in the introductory section, all accumulation points $\bar x$ (in norm) of the sequence of iterates $(x_k)$ generated by the Proximal Newton method globalized via \eqref{eq:dampedstep} are stationary points of problem \eqref{eq:prob}.
\end{theorem}
\begin{proof}
	Let us consider a modified version of our minimization problem as in \eqref{eq:modprob} in Lemma~\ref{lem:Fschlange} and choose $q(x) = \frac 12 Q(x)^2$ for $Q: X \times X \to \RR$ such that $\tilde g = g + q$ is (strongly) convex on its domain. 
	
	This is always possible by \eqref{eq:kappa2}. According to Lemma~\ref{lem:Fschlange}, the sequence of iterates remains unchanged and step computation takes the form
	\begin{align*}
	x_{k+1} = \tilde x_{k+1} = \argmin_{y \in X} \tilde g (y) + \frac 12 \big( H_{x_k} - Q + \omega \mathcal{R} \big)(y)^2 - \big((H_{x_k} + \omega_k \mathcal{R})x_k - f'(x_k) \big)y
	\end{align*}
	with first order optimality conditions
	\begin{align*}
	H_{x_k}(x_k) + \omega_k \mathcal{R}x_k - f'(x_k) \in \partial \tilde g (x_{k+1}) + \big( H_{x_k} - Q + \omega \mathcal{R} \big)(x_{k+1})
	\end{align*}
	where $\partial \tilde g (x_{k+1})$ denotes the convex subdifferential of $\tilde g$ at $x_{k+1}$. Consequently, we know that there exists some $\tilde \eta_k \in \partial \tilde g (x_{k+1})$ such that 
	\begin{align*}
	\tilde \eta_k + \big( f'(x_{k+1}) - Q x_{k+1} \big) = \tilde r_{x_k}\big(\Delta x_k (\omega_k)\big)
	\end{align*}
	with the remainder term on the right-hand side given by
	\begin{align*}
		\tilde r_x(v)\coloneqq f'(x+v)-f'(x)- \big(H_x + \omega \mathcal{R}\big)v
	\end{align*}
	holds. As before, the remainder term $\tilde r_{x_k}\big(\Delta x_k (\omega_k)\big) = r_{x_k}\big(\Delta x_k (\omega_k)\big)$ tends to zero for $k \to \infty$, i.e., we have $\tilde \eta \coloneqq \lim_{k \to \infty} \tilde \eta_k = -f'(\bar x) + Q \bar x$. The definition of the convex subdifferential $\partial \tilde g$ together with the lower semi-continuity of $\tilde g$ directly yields
	\begin{align*}
	\tilde g (u) - \tilde g(\bar x) &= g(u) + \frac 12 Q(u)^2 - g(\bar x) - \frac 12 Q(\bar x)^2 \\
	&\geq g(u) + \frac 12 Q(u)^2 - \liminf_{k \to \infty} g(x_k) - \lim_{k \to \infty} \frac 12 Q(x_k)^2 \\
	&= \liminf_{k \to \infty} \tilde g(u) - \tilde g(x_k) \\
	&\geq \liminf_{k \to \infty} \tilde \eta_k (u- x_k) = \lim_{k \to \infty} \tilde \eta_k (u- x_k) = \tilde \eta (u-\bar x)
	\end{align*}
	for any $u \in X$ which proves the inclusion $\tilde \eta \in \partial \tilde g(\bar x)$. The evaluation of the latter limit expression can easily be retraced by splitting
	\begin{align} \label{eq:weaklim}
	\tilde \eta_k (u- x_k) = \tilde \eta_k (u-\bar x) + (\tilde \eta_k - \tilde \eta)(\bar x- x_k) + \tilde \eta (\bar x- x_k) \, .
	\end{align}
	In particular, we recognize $\tilde \eta \in \partial \tilde g(\bar x)$ as $-f'(\bar x) + Q \bar x \in \partial \tilde g(\bar x)$ and equivalently $-f'(\bar x) \in \partial_F g(\bar x)$ for the Frech\'et-subdifferential $\partial_F$. This implies $0 \in \partial_F F (\bar x)$, i.e., the stationarity of our limit point $\bar x$. \qed
\end{proof}

Also note that in general the above global convergence result does not rely on the strong convexity of the composite objective function $F$ but yields stationarity of limit points also in the non-convex case of $\ksum < 0$ and $\omega_k > -(\ksum)$ chosen adequately. In particular, this ensures that also independent of strong convexity assumptions near optimal solutions, the algorithm approaches the optimal solution and can then benefit from additional convexity at later iterations.

While bounded sequences in finite dimensional spaces always have convergent subsequences, we can only expect \emph{weak subsequential convergence} in general Hilbert spaces in this case. As one consequence, existence of minimizers of nonconvex functions on Hilbert spaces can usually only be established in the presence of some compactness. On this count we note that in \eqref{eq:weaklim} even weak convergence of $x_k \rightharpoonup \bar x$ would be sufficient. Unfortunately, in the latter case we cannot evaluate $f'(x_k) \to f'(\bar x)$. 

In order to extend our proof to this situation, we require some more structure for both of the parts of our composite objective functional. To this end, we remember the following well-known definition of compact operators:
\begin{definition}
	A linear operator $K:X \to Y$ between two normed vector spaces $X$ and $Y$ is called compact if one of the following equivalent statements holds:
	\begin{itemize}
		\item[1)] The image of the unit ball of $X$ is relatively compact in $Y$ (, i.e., its closure is compact).
		\item[2)] For any bounded sequence $(x_n)_{n \in \NN} \subset X$ the image sequence $(Kx_n)_{n \in \NN} \subset Y$ contains a strongly convergence subsequence $\big(x_{n_k} \big)_{k \in \NN} \subset X$.
	\end{itemize}
\end{definition}
With this notion at hand, we can formulate the following global convergence theorem:
\begin{theorem}
	Let $f$ be of the form $f(x)= \hat f(x) + \check f(Kx)$ where $K$ is a compact operator. Additionally, assume that $g + \hat f$ is convex and weakly lower semi-continuous in a neighborhood of stationary points of \eqref{eq:prob}. Then weak convergence of the sequence of iterates $x_k \rightharpoonup \bar x$ suffices for $\bar x$ to be a stationary point of \eqref{eq:prob}.
	
	If $F$ is strictly convex and radially unbounded, the whole sequence $x_k$ converges weakly to the unique minimizer $x_*$ of $F$. If $F$ is $\kappa$-strongly convex, with $\kappa > 0$, then $x_k \to x_*$ in norm. 
\end{theorem}
\begin{proof}
	We can employ the same proof as above replacing $g$ by $g+\hat f$ and using that $\tilde f'(Kx_k) \to \check f'(K\bar x)$ in norm, by compactness. This then shows finally
	\begin{align*}
	(g+\hat f)(u)-(g+\hat f)(\bar x) \ge \eta (u-\bar x) \, ,
	\end{align*}
	i.e., $\eta = -\check f'(K\bar x)K \in \partial(g+\hat f)(\bar x) = \partial_F g(\bar x)+\{\hat f'(\bar x)\}$ which in particular implies 
	\begin{align*}
	- f'(\bar x) = -\check f'(K\bar x)K - \hat f'(\bar x) \in \partial_F g(\bar x) \, .
	\end{align*}
	This again constitutes $0 \in \partial_F F(\bar x)$ and thereby the stationarity of the weak limit point $\bar x$.
	
	Let us now consider the second assertion: $F$ being strictly convex as well as radially unbounded yields that problem \eqref{eq:prob} has a unique solution $x_*$. Additionally, we know that our sequence of iterates is bounded as a consequence of which we can select a weakly convergent subsequence. The first assertion of the theorem then implies that the limit of each subsequence we choose is a stationary point of problem \eqref{eq:prob}, and thus by convexity to the unique optimal solution $x_*$. A standard argument then shows that the whole sequence converges to $x_*$ weakly. 
	
	If $F$ is $\kappa$-strongly convex, then as discussed below~\eqref{eq:kappa2} the diameter the level sets $L_{F(x_k)}$ tends to $0$ as $k\to \infty$, since $F(x_k)\to F(x_*)$. This implies $\|x_k-x_*\|_X \to 0$. \qed
\end{proof}

\section{Second order semi-smoothness} \label{sec:soss}
In order to be able to benefit from the local acceleration result in Theorem~\ref{thm:localacc}, we have to ensure that under the assumptions on $F$ stated in Section~\ref{sec:intro} eventually also full steps are admissible for sufficient descent according to our criterion formulated in \eqref{eq:lambda}. To this end, we want to introduce a new notion of differentiability, which we call second order semi-smoothness, and investigate how it interacts with our Proximal Newton method.

For the smooth part $f$ of our composite objective function $F$ we define a second order semi-smoothness property at some $x_* \in \mathrm{dom} f$ by
\begin{align} \label{eq:sos}
f(x_* + \xi) = f(x_*) + f'(x_*)\xi + \frac 12 H_{x_* + \xi}(\xi,\xi) + o(\norm{\xi}{X}^2)  \quad \mbox{for } \xi \to 0.
\end{align}
for any $\xi \in X$. This will be precisely the assumption that we need to conclude transition to fast local convergence in the following section. 

We give a general definition for operators. Denote by $L^{(2)}(X,Y)$ the normed space of bounded vector valued bilinear forms $X\times X \to Y$, equipped with usual norm:
\[
\|B\|_{L^{(2)}(X,Y)} = \sup_{\xi_1,\xi_2\neq 0} \frac{\|B(\xi_1,\xi_2)\|_Y}{\|\xi_1\|_X\|\xi_2\|_X}. 
\]
\begin{definition}
	Let $X,Y$ be normed linear spaces and  let $D\subset X$ be a neighborhood of $x_*$. Consider  a continuously differentiable operator $T : D \to Y$, and a bounded mapping 
	\[
	T'' : D \to L^{(2)}(X,Y). 
	\]
	We call $T$ second order semi-smooth at $x_* \in X$ with respect to $T''$, if the following estimate holds:
	\[
	\|T(x_*+\xi)-T(x_*)-T'(x_*)\xi-\frac12 T''(x_*+\xi)(\xi,\xi)\|_Y = o(\|\xi\|_X^2) \quad \text{for } \xi \to 0
	\]
\end{definition}
Since $T''$ is evaluated at $x_*+\xi$, the choice of $T''$ is far from unique.  
Twice continuously differentiable operators apparently are second order semi-smooth:
\begin{proposition} \label{prop:contdiffsos}
	Assume that $T$ is twice continuously differentiable at $x_*$. 
	Then $T$ is second order semi-smooth at $x_*$ with respect to the ordinary second derivative $T''$.
\end{proposition}
\begin{proof}
	This follows by a simple computation: 
	\begin{align*}
	T(x_*)&+T'(x_*) \xi +\frac{1}{2}T''(x_* + \xi)(\xi,\xi)\\
	&=\big[T(x_*)+T'(x_*) \xi+\frac{1}{2}T''(x_*)(\xi,\xi)\big]+\frac{1}{2}\big[T''(x_* + \xi)(\xi,\xi)-T''(x_*)(\xi,\xi)\big]
	\end{align*}
	Both terms in square brackets are $o(\norm{\xi}{X}^2)$. The first by Fr\'echet differentiability of $T$, the second by continuity of $T''(x)$. \qed
\end{proof}
It is an obvious remark that the sum of two second order semi-smooth functions is second order semi-smooth again with linear and quadratic terms defined
via sums. Furthermore, the following chain rule can be shown: 
\begin{theorem}\label{th:chainrule}
	Suppose that $S:D_S\to Y$ and $T:D_T\to Z$ with $S(D_S)\subset D_T$ are second order semi-smooth at $x_* \in D_S$ and $y_*=S(x_*)$ with respect to $S''$ and $T''$, respectively. Then $T\circ S$ is second order semi-smooth with respect to $(T\circ S)''$, defined as follows:
	\[
	(T\circ S)''(x)(\xi_1,\xi_2) \coloneqq  T''(y)(S'(x) \xi_1,S'(x)\xi_2)+T'(y)S''(x)(\xi_1,\xi_2).
	\]
\end{theorem}  
\begin{proof}
	We introduce the notations $y_*=S(x_*)$, $x=x_*+\xi$, $y=S(x)$, and  $\eta = y-y_*$.
	With these prerequisites we can, as usual for chain rules, split the remainder term:
	\begin{align}
	\notag (T\circ S)(x)&-(T\circ S)(x_*)-(T\circ S)'(x_*)\xi-\frac12 (T\circ S)''(x)(\xi,\xi)\\
	\notag &=T(y)-T(y_*) - T'(y_*)S'(x_*)\xi \\
	\notag &\quad - \frac12\Big(T''(y)(S'(x) \xi,S'(x) \xi) + T'(y) S''(x)(\xi,\xi)\Big)\\
	\label{eq:s1} &=T(y)-T(y_*)-T'(y_*)\eta-\frac{1}{2}T''(y) (\eta,\eta)\\
	\label{eq:s2} &\quad + T'(y_*) \left(S(x)-S(x_*)- S'(x_*)\xi -\frac{1}{2}S''(x)(\xi,\xi)\right)\\
	\label{eq:s3}&\quad + \frac12( T'(y_*) - T'(y))S''(x)(\xi,\xi)\\
	\label{eq:s4} &\quad + \frac{1}{2}\left(T''(y) (\eta,\eta)-T''(y) (S'(x) \xi,S'(x) \xi)\right)
	\end{align}
	We will show that each of the expressions \eqref{eq:s1}-\eqref{eq:s4} is $o(\norm{\xi}{X}^2)$. For $\eqref{eq:s1}$ this follows from second order semi-smoothness of $T$, while second order semi-smoothness of $S$ implies the desired result for \eqref{eq:s2}. Continuity of $T'$ and boundedness of $S''$ yield that \eqref{eq:s3} is $o(\norm{\xi}{X}^2)$. 
	Finally, \eqref{eq:s4} can be reformulated via the third binomial formula:
	\begin{align*}
	\|T''(y) (\eta,\eta)-T''(y) (S'(x) \xi,S'&(x) \xi)\|_Z=\|T''(y) (\eta+S'(x) \xi,\eta-S'(x) \xi)\|_Z \\
	&\le \|T''(y)\|_{L^{(2)}(Y,Z)}\|\eta+S'(x) \xi\|_Y\|\eta-S'(x) \xi\|_Y.
	\end{align*}
	By continuous differentiablity of $S$ (which is a prerequisite of second order semi-smoothness by our definition) we estimate:
	\begin{align}
	\label{eq:f2}  \norm{\eta + S'(x) \xi}{Y} &=O(\norm{\xi}{X})\\
	\label{eq:f2b} \norm{\eta-S'(x) \xi}{Y}\le\norm{\eta-S'(x_*) \xi}{Y}+\norm{(S'(x_*)-S'(x)) \xi}{Y}  &=o(\norm{\xi}{X}),
	\end{align}
	which finally yields the desired result. \qed
\end{proof}
\begin{remark}
	In the case $T'(y_*)=0$, we observe from \eqref{eq:s2} that $S$ only needs to be continuously differentiable and we may set $S''=0$. 
\end{remark}
Second order semi-smoothness of $T$ and semi-smoothness of $T'$ as in \eqref{eq:semismooth} are closely related but are not equivalent in general. Even in the case of $T''(x) \coloneqq \partial_N T'(x)$ we cannot conclude one condition from the other, e.g. via the fundamental theorem of calculus, because of the lack of continuity of $\partial_N T'$. 

Let us shortly give a both simple and illustrative example: Consider the function
\begin{align*}
	h: \, \RR \to \RR \, , \, x \mapsto x^3 \sin \big( \frac 1x \big)
\end{align*}
which is continuously differentiable with $h'(x) = x\big[3x\sin\big(\frac 1x\big) - \cos \big( \frac 1x \big)\big]$, $x \neq 0$, and $h'(0) = 0$. The cubic asymptotics of $h$ suggest that $T''(x) \equiv 0$ is a possible definition for second order semi-smoothness of $h$ at $x_* = 0$ as above. Apparently, we obtain for $x\in \RR$ and $\delta x = x - x_* = x$:
\begin{align*}
	|h(x) - h(x_*) - h'(x_*)\delta x - \frac 12 T''(x)(\delta x)^2| = |\delta x|^3\big|\sin\big( \frac 1x \big)\big| = O\big( |\delta x|^3 \big) \, , \, \delta x \to 0 \, ,
\end{align*}
i.e., that $h$ is indeed second order semi-smooth at $x_* = 0$ with respect to $T''$. On the other hand, we have 
\begin{align*}
	|h'(x_*) - h'(x) - T''(x)(x_*-x)| = |\delta x| \big| 3x \sin \big(\frac 1x\big) - \cos \big( \frac 1x \big) \big| \neq o\big( \big| \delta x \big| \big) \, , \, \delta x \to 0 \, ,
\end{align*}
which implies that $h'$ is indeed not semi-smooth at $x_* = 0$ with respect to the same $T''$, cf. \eqref{eq:semismooth}. However, in many cases of practical interest, both conditions can be shown to hold.

For instance, the function $\phi(x)= \max\{0,x\}^2$ is second order semi-smooth at the point $x=0$ with respect to 
\[
\phi''(\xi) = \left\{
\begin{array}{rcl}
0 &:& \xi < 0\\
1 &:& \xi \ge 0
\end{array}\right. 
\]
as well as twice Fr\'echet differentiable (and thus also second-order semi-smooth, cf. Proposition~\ref{prop:contdiffsos}) at any other point $x\neq 0$ with the same $\phi''(\xi)$. By standard techniques we can lift this property to superposition operators on $L_p$-spaces for appropriate $p$.

For convenience, we recapitulate the following lemma, which is a slight generalization of a standard result on continuity of superposition operators. 
\begin{lemma}\label{lem:kras}
	Let $\Omega$ a measurable subset of $\RR^d$, and 
	$\psi: \RR \times \Omega \to \RR$.
	For each measurable function $x:\Omega \to \RR$ assume that the function $\Psi(x)$, defined by
	$\Psi(x)(t) = \psi(x(t),t)$ is measurable. 
	Let $x_* \in L_p(\Omega,\RR)$ be given. Then the following assertion holds:
	
	If $\psi$ is continuous with respect to $x$ at $(x_*(t),t)$ for almost all  $t\in \Omega$, 
	and $\Psi$ maps $L_p(\Omega,\RR)$ into $L_s(\Omega,\RR)$ for $1 \le p,s < \infty$, 
	then $\Psi$ is continuous at $x_*$ in the norm topology.
\end{lemma}
\begin{proof}
	cf. e.g. \cite[Lemma 3.1]{Schiela2006}. \qed
\end{proof}
The standard text book result requires $\psi$ to be a Caratheodory function, and thus in particular continuous in $x$ for all $t\in \Omega$. This assumption, is slightly weakened here to the almost everywhere sense. It is known, for example, that pointwise limits and suprema of Caratheodory functions yield superposition operators that map measurable functions to measurable functions. The mapping $\phi''$ as defined above is an example. Importantly, this result is not true for the case $p<s=\infty$. 

\begin{proposition}\label{pro:soss}
	Consider a real function $\phi:\RR \to \RR$ with globally Lipschitz-continuous derivative $\phi' : \RR \to \RR$, which is second order semi-smooth with respect to a bounded function $\phi'' : \RR \to \RR$. Let $\Omega \subset \RR^d$ be a set of finite measure and assume that the composition $\phi''\circ u$ is measurable for any measurable function $u : \Omega \to \RR$. Let $p>2$. Then for each $x\in L_p(\Omega)$ the superposition operator $\Phi : L_p(\Omega) \to L_1(\Omega)$ is second order semi-smooth with respect to $\Phi''(x)\in L_2(L_p(\Omega),L_1(\Omega))$ defined by $\Phi''(x)(\xi_1,\xi_2)(\omega)=\phi''(x(\omega))\xi_1(\omega)\xi_2(\omega)$ almost everywhere. 
\end{proposition}
\begin{proof}
	Consider a representative of $x\in L_p(\Omega)$ and the function 
	\[
	r_x(\omega,t) \coloneqq  \frac{\phi(x(\omega)+t)-\phi(x(\omega))-\phi'(x(\omega))t-\phi''(x(\omega)+t)t^2}{t^2}
	\]
	which is defined for $t\neq 0$ and $r_x(\omega,t)\coloneqq 0$ for $t=0$. By Lipschitz-continuity of $\phi'$ and boundedness of $\phi''$ we observe that $r_x$ is bounded uniformly on $\Omega \times \RR$. 
	Thus, the superposition operator $R_x : L_p(\Omega) \to L_s(\Omega)$ :  $R_x(\xi)(\omega)=r_x(\omega,\xi(\omega))$ is well defined for any $1 \le s\le \infty$. By second order semi-smoothness $r_x(\omega,\cdot)$ is continuous at $t=0$ for almost all $\omega\in \Omega$. Hence, by Lemma~\ref{lem:kras} $R_x$ is continuous as an operator at $\xi=0$ for any $s<\infty$. By the H\"older inequality with $1/s+2/p=1$ we conclude the desired estimate:
	\begin{align*}
	\|\Phi(x+\xi)-\Phi(x)-\Phi'(x)\xi-\Phi''(x)(\xi,\xi)\|_{L_1(\Omega)} &= \|R_x(\xi)\cdot\xi \cdot \xi\|_{L_1(\Omega)}\\
	\le \|R_x(\xi)\|_{L_s(\Omega)}\|\xi\|^2_{L_p(\Omega)} &= o(\|\xi\|_{L_p(\Omega)}^2). 
	\end{align*} \qed
\end{proof}
Unsurprisingly and in analogy to the theory of semi-smooth superposition operators, there is a norm gap in the sense that Proposition~\ref{pro:soss} is false for $p=2$. This is closely related to the so call two-norm discrepancy (cf. e.g. \cite{Tro2010}). 

As in the above example, $\phi''(\xi)$ has a discontinuity at $\xi=0$, so we cannot expect that $\Phi''$ is a continuous mapping on a given open set. However, we can show the following result:
\begin{proposition}\label{pro:cont}
	Let $p>2$ and $x_* \in L_p(\Omega)$ be fixed. Assume that function $(\omega,t) \to \phi''(x_*(\omega)+t)$ is continuous in $t$ for almost all $\omega \in \Omega$. Then the mapping $\Phi'' : L_p(\Omega) \to L^{(2)}(L_p(\Omega),L_1(\Omega))$ is continuous at $x_*$. 
\end{proposition}
\begin{proof}
	We apply Lemma~\ref{lem:kras} to the superposition operator $\tilde \Phi''(x)(\omega)\coloneqq \phi''(x(\omega))$, which maps $L_p(\Omega)\to L_s(\Omega)$ and the use the H\"older inequality to conclude:
	\[
	\|[\Phi''(x)-\Phi''(x_*)](\xi_1,\xi_2)\|_{L_1(\Omega)} \le \|\tilde \Phi''(x)-\tilde \Phi''(x_*)\|_{L_s(\Omega)}\|\xi_1\|_{L_p(\Omega)}\|\xi_2\|_{L_p(\Omega)}.
	\]
	\qed
\end{proof}
In our example $\phi(x)=\max\{0,x\}^2$ fulfills the hypothesis of this theorem at $x_*\in L_p(\Omega)$, if $x_*(\omega)=0$ only on a set of measure $0$ in $\Omega$. This kind of regularity assumption can also be found frequently in the literature on semi-smooth Newton methods (cf. e.g. \cite{HintUlb2004}).  

\section{Transition to Fast Local Convergence} \label{sec:transition}

Let us now turn our attention back to our Proximal Newton method and consider the admissibility of undamped update steps near optimal solutions of problem \eqref{eq:prob}. Both the semi-smoothness of $f'$ from \eqref{eq:semismooth} and the second order semi-smoothness of $f$ from \eqref{eq:sos} will contribute a crucial part to the proof of this result. Additionally, the local acceleration result from Theorem~\ref{thm:localacc} will play an important role.

However, an algorithm that tests in every iterate, whether the undamped Newton step is acceptable is likely to compute many unnecessary trial iterates during the early phase of globalization. Thus, it is of interest, whether damped Newton steps are acceptable as well close to the solution. 

In order to establish the corresponding proposition of admissibility we will first have to investigate the relation between damped and undamped steps more closely.
\begin{lemma} \label{lem:steprelations}
	Let $H_x$ be a bilinear form as in \eqref{eq:kappa1} and assume that $g$ suffices \eqref{eq:kappa2} where $\ksum > 0$ holds and $x \in X$ is arbitrary. Then the damped update step $\Delta x(\omega)$ from \eqref{eq:dampedstep} and the undamped update step $\Delta x$ from \eqref{eq:fullstep} satisfy the estimates
	\begin{align}
	\norm{\Delta x - \Delta x(\omega)}{X} &\leq \frac{\omega}{\ksum} \norm{\Delta x(\omega)}{X} \label{eq:steprelations1} \\
	\norm{\Delta x(\omega)}{X} \leq \norm{\Delta x}{X} &\leq \big( \frac{\omega}{\ksum} + 1 \big) \norm{\Delta x(\omega)}{X} \label{eq:steprelations2} 
	\end{align}
	for any $\omega \geq 0$. 
\end{lemma}
\begin{proof}
	The above set of estimates can all be deduced from adequate proximal representations of the respective update steps. We can characterize the undamped step via $\Delta x = x_+ - x$ where the updated iterate is given by
	\begin{align*}
	x_+ &= \argmin_{y \in X} f'(x)(y-x) + \frac{1}{2} H_x(y-x,y-x)+g(y)-g(x) \\
	&= \p_g^{H_x} \big(H_x (x)- f'(x)\big) \, .
	\end{align*}
	Now, consider the corresponding inequality from Proposition~\ref{prop:scndprox} for $\varphi = H_x (x)- f'(x)$, $H = H_x$ and $\xi \coloneqq x_+(\omega)$ given by
	\begin{align*}
	\big[ H_x(x) - f'(x) - H_x(x_+)\big] (x_+(\omega) - x_+) \leq g(x_+(\omega)) - g(x_+) - \frac{\kzwei}{2}\norm{x_+(\omega) - x_+}{X}^2
	\end{align*}
	which can be rearranged to a more useful form via
	\begin{align} \label{eq:stepconv1}
	\big[ H_x(\Delta x) + f'(x) \big](\Delta x - \Delta x(\omega)) \leq g(x_+(\omega)) - g(x_+) - \frac{\kzwei}{2}\norm{\Delta x - \Delta x(\omega)}{X}^2 \, .
	\end{align}
	For the damped update step we want to consider a different form than in \eqref{eq:dampprox} and attribute the additional norm term $\frac{\omega}{2}\norm{\cdot}{X}^2$ to the lower argument function $g$. This results in the proximal representation
	\begin{align*}
	x_+(\omega) = \p_{g +\frac{\omega}{2}\norm{\cdot}{X}^2 }^{H_x} \big(H_x (x) + \omega \mathcal R x - f'(x)\big) \, .
	\end{align*}
	The deduction of the respective inequality induced by the first order conditions of the proximal subproblem will turn out to be slightly more complicated. We use $H = H_x$ and $\varphi = H_x (x) + \omega \mathcal R x - f'(x)$ together with $\xi = x_+$ in Proposition~\ref{prop:scndprox}. Note here that the lower argument function $g +\frac{\omega}{2}\norm{\cdot}{X}^2$ satisfies \eqref{eq:kappa2} with constant $\kzwei + \omega$. Thus, we obtain
	\begin{align} \label{eq:stepconv2start}
	\begin{split}
	\big[ - H_x \big(&\Delta x (\omega)\big) + \omega \mathcal{R} x - f'(x) \big]\big(\Delta x - \Delta x(\omega) \big) \\
	&\leq g(x_+) - g(x_+(\omega)) + \frac{\omega}{2}\big( \norm{x_+}{X}^2 - \norm{x_+(\omega)}{X}^2 \big) - \frac{\kzwei + \omega}{2}\norm{\Delta x - \Delta x(\omega)}{X}^2 \, .
	\end{split}
	\end{align}
	We bring the Riesz-term $\omega \mathcal{R} \big(x , \Delta x - \Delta x(\omega) \big)$ to the right-hand side of \eqref{eq:stepconv2start} and recognize
	\begin{align*}
	\norm{x_+}{X}^2 - \norm{x_+(\omega)}{X}^2 - 2\mathcal{R} \big(x , \Delta x - \Delta x(\omega) \big) = \norm{\Delta x}{X}^2 - \norm{\Delta x (\omega)}{X}^2
	\end{align*}
	which results in
	\begin{align} \label{eq:stepconv2mitte}
	\begin{split}
	\big[ - H_x \big(&\Delta x (\omega)\big) - f'(x) \big]\big(\Delta x - \Delta x(\omega) \big) \\
	&\leq g(x_+) - g(x_+(\omega)) + \frac{\omega}{2}\big( \norm{\Delta x}{X}^2 - \norm{\Delta x(\omega)}{X}^2 \big) - \frac{\kzwei + \omega}{2}\norm{\Delta x - \Delta x(\omega)}{X}^2 \, .
	\end{split}
	\end{align}
	This inequality will be of importance once more later on. For now, we estimate the term
	\begin{align*}
	\norm{\Delta x}{X}^2 &- \norm{\Delta x(\omega)}{X}^2 - \norm{\Delta x - \Delta x(\omega)}{X}^2 \\
	&= \big( \norm{\Delta x}{X} + \norm{\Delta x(\omega)}{X} \big)\big( \norm{\Delta x}{X} - \norm{\Delta x(\omega)}{X} \big) - \norm{\Delta x - \Delta x(\omega)}{X}^2 \\
	&\leq \norm{\Delta x - \Delta x(\omega)}{X}\big( \norm{\Delta x}{X} + \norm{\Delta x(\omega)}{X} - \norm{\Delta x - \Delta x(\omega)}{X} \big) \\
	&\leq 2\norm{\Delta x(\omega)}{X}\norm{\Delta x - \Delta x(\omega)}{X}
	\end{align*}
	such that \eqref{eq:stepconv2mitte} takes the form
	\begin{align} \label{eq:stepconv2}
	\begin{split}
	\big[ - H_x \big(&\Delta x (\omega)\big) - f'(x) \big]\big(\Delta x - \Delta x(\omega) \big) \\
	&\leq g(x_+) - g(x_+(\omega)) + \omega \norm{\Delta x(\omega)}{X}\norm{\Delta x - \Delta x(\omega)}{X} - \frac{\kzwei}{2}\norm{\Delta x - \Delta x(\omega)}{X}^2 \, .
	\end{split}
	\end{align}
	Now, we add \eqref{eq:stepconv1} and \eqref{eq:stepconv2} which yields
	\begin{align*}
	H_x \big(\Delta x - \Delta x(\omega) \big)^2 \leq \omega \norm{\Delta x(\omega)}{X}\norm{\Delta x - \Delta x(\omega)}{X} - \kzwei \norm{\Delta x - \Delta x (\omega)}{X}^2 \, .
	\end{align*}
	Here we can use assumption \eqref{eq:kappa1} on $H_x$ and rearrange the resulting estimate to
	\begin{align*}
	\norm{\Delta x - \Delta x(\omega)}{X}^2 \leq \frac \omega {\ksum} \norm{\Delta x(\omega)}{X}\norm{\Delta x - \Delta x(\omega)}{X} \, .
	\end{align*}
	This is exactly the first asserted inequality \eqref{eq:steprelations1} if we divide by $\norm{\Delta x - \Delta x(\omega)}{X}$ which we can assume to be non-zero without loss of generality. From here, we can directly deduce the second part of \eqref{eq:steprelations2} since we can take advantage of \eqref{eq:steprelations1} by
	\begin{align*}
	\norm{\Delta x}{X} - \norm{\Delta x (\omega)}{X} \leq \norm{\Delta x - \Delta x(\omega)}{X} \leq \frac \omega {\ksum} \norm{\Delta x(\omega)}{X} \, .
	\end{align*}
	The first part of \eqref{eq:steprelations2} on the other hand requires some more consideration. We start at \eqref{eq:stepconv2mitte} but now take another route and directly add it to \eqref{eq:stepconv1} which yields
	\begin{align*}
	H_x\big(\Delta x - \Delta x(\omega) \big)^2 + \big( \kappa_2 + \frac \omega 2 \big)\norm{\Delta x - \Delta x(\omega)}{X}^2 \leq \frac \omega 2 \big( \norm{\Delta x}{X}^2 - \norm{\Delta x(\omega)}{X}^2 \big)
	\end{align*}
	and thereby 
	\begin{align} \label{eq:stepconv3}
	0 \leq \big( \keins + \kzwei + \frac \omega 2 \big)\norm{\Delta x - \Delta x(\omega)}{X}^2 \leq \frac \omega 2 \big( \norm{\Delta x}{X}^2 - \norm{\Delta x(\omega)}{X}^2 \big)
	\end{align}
	as we use \eqref{eq:kappa1} for $H_x$. All prefactors in \eqref{eq:stepconv3} are positive due to our assumptions such that the first part of \eqref{eq:steprelations2} follows. This completes the proof. \qed
\end{proof}
The equivalence result for damped and undamped update steps in the form of \eqref{eq:steprelations2} enables the proof of the following Corollary which will turn out to be useful for the admissibility of damped steps close to optimal solutions.
\begin{corollary} \label{cor:orelations}
	Close to an optimal solution $x_*$ of \eqref{eq:prob} we can find constants $c_1 , c_2 > 0$ such that the following estimates hold:
	\begin{align*}
	\norm{x_+(\omega) - x_*}{X} \leq c_1 \norm{x-x_*}{X} \, , \, \norm{x-x_*}{X} \leq c_2 \norm{\Delta x (\omega)}{X}
	\end{align*}
\end{corollary}
\begin{proof}
	For the deduction of both asserted inequalities, we will take advantage of the local superlinear convergence stated in Theorem~\ref{thm:localacc}, i.e., $\norm{x_+ - x_*}{X} = o\big( \norm{x-x_*}{X} \big)$ in the limit of $x \to x_*$. Consequently, we can write
	\begin{align} \label{eq:omittelspsi}
	\norm{x_+ - x_*}{X} = \psi \big( \norm{x-x_*}{X} \big) \norm{x-x_*}{X}
	\end{align}
	for some function $\psi:[0,\infty[ \to [0,\infty[$ with $\psi(t) \to 0$ for $t \to 0$. With this helpful representation at hand, we estimate
	\begin{align*}
	\norm{x_+(\omega) - x_*}{X} &\leq \norm{x-x_*}{X} + \norm{\Delta x (\omega)}{X} \leq \norm{x-x_*}{X} + \norm{\Delta x}{X} \\
	&\leq 2\norm{x-x_*}{X} + \norm{x_+ - x_*}{X} = \big[ 2 + \psi \big( \norm{x-x_*}{X} \big) \big] \norm{x-x_*}{X} \, .
	\end{align*}
	By the definition of $\psi$ above, this directly implies the first asserted inequality. We can deduce the second one similarly quickly via
	\begin{align*}
	\norm{x-x_*}{X} \leq \norm{x_+ - x_*}{X} + \norm{\Delta x}{X} = \psi \big( \norm{x-x_*}{X} \big) \norm{x-x_*}{X} + \norm{\Delta x}{X} \, .
	\end{align*}
	We can assume $\psi \big( \norm{x-x_*}{X} \big) < 1$ close to the optimal solution $x_*$ and thereby deduce
	\begin{align*}
	\norm{x-x_*}{X} &\leq \big[ 1 -\psi \big( \norm{x-x_*}{X} \big)  \big]^{-1} \norm{\Delta x}{X} \\
	&\leq \big[ 1 -\psi \big( \norm{x-x_*}{X} \big)  \big]^{-1} \big( \frac{\omega}{\ksum} + 1 \big) \norm{\Delta x(\omega)}{X}
	\end{align*}
	with the additional help of \eqref{eq:steprelations2}. Taking into account that $\omega$ remains bounded completes the proof of the second asserted inequality. \qed
\end{proof}
Now we are in the position to prove the admissibility of both undamped and damped steps close to optimal solutions of the composite minimization problem \eqref{eq:prob}. We will see that undamped steps will generally be admissible whereas for the admissibility of damped steps we will have to assume an additional property of the second order model bilinear forms $H_x$.
\begin{proposition} \label{prop:dampedstepadmiss}
	Let $x_* \in X$ be an optimal solution of \eqref{eq:prob} and let $H_x \in \partial_N f'(x)$ suffice \eqref{eq:kappa1} as well as $g$ suffice \eqref{eq:kappa2} with $\ksum > 0$ in a neighborhood of $x_*$. Additionally, suppose that \eqref{eq:sos} holds for $f$ as well as \eqref{eq:semismooth} holds for $f'$ at $x_*$. 
	
	Steps as in \eqref{eq:dampedstep} for any $\omega \ge 0$ are admissible for sufficient descent according to \eqref{eq:lambda} for any $\gamma < 1$ if the second order bilinear forms $H_x$ satisfy a bound of the form
	\begin{align} \label{eq:Hcontass}
	(H_{x_+(\omega)}-H_x) (x_+(\omega)-x_*)^2 = o\big( \norm{x-x_*}{X}^2 \big) \, \mbox{ for } x \to x_*.
	\end{align}
	
	In particular:
	\begin{itemize}
		\item[i)] 	full steps $\Delta x$ as defined in \eqref{eq:fullstep} are eventually admissible.
		\item[ii)] 	if the mapping $x \mapsto H_x$ is continuous at $x=x_*$, then eventually all steps are admissible. 
	\end{itemize}

\end{proposition}
\begin{proof}
	Let us take a look at the descent in the composite objective function $F$ when performing an update step and see which estimates we can deduce with the help of the assumptions and results preceding this proposition. 
	
	We will denote the update by $\Delta x (\omega)$ or $x_+(\omega) = x + \Delta x (\omega)$ respectively for some arbitrary $\omega \geq 0$ such that the notation comprises both the damped and undamped case for the update step. Now, we write 
	\begin{align*}
	F(x+\Delta x(\omega)) - F(x) = f(x+\Delta x(\omega)) - f(x) + g(x+\Delta x(\omega)) - g(x)
	\end{align*}
	and estimate the descent in the smooth part of the objective function $f(x+\Delta x(\omega)) - f(x)$. By telescoping we obtain the following identity:
	\begin{align}
	&f(x_+(\omega)) - f(x) - f'(x)\Delta x(\omega) - \frac 12 H_x (\Delta x(\omega))^2 \nonumber \\
	&= f(x_+(\omega))-f(x_*)-f'(x_*)(x_+(\omega)-x_*)-\frac12 H_{x_+(\omega)}(x_+(\omega)-x_*)^2 \nonumber \\
	&\quad+ f(x_*)+f'(x_*)(x- x_*) + \frac12 (H_{x_+(\omega)}-H_x) (x_+(\omega)-x_*)^2 \nonumber \\
	&\quad- f(x)  - H_x (\Delta x(\omega))^2+\frac12 (H_x(x_*-x_+(\omega))^2+ H_x (\Delta x(\omega))^2) \nonumber \\
	&\quad-f'(x)\Delta x(\omega)+f'(x_*)\Delta x(\omega)+H_x(x-x_*,\Delta x(\omega)) \nonumber \\
	&\quad +H_x(x_*-x_+(\omega)+\Delta x(\omega) ,\Delta x(\omega))\nonumber \\
	&=\left[f(x_+(\omega))-f(x_*)-f'(x_*)(x_+(\omega)-x_*)-\frac12 H_{x_+(\omega)}(x_+(\omega)-x_*)^2\right]\nonumber \\
	&\quad -\left[f(x) - f(x_*) - f'(x_*)(x-x_*) - \frac12 H_x(x-x_*)^2\right] \nonumber \\
	&\quad -\Big[(f'(x)-f'(x_*))\Delta x(\omega)-H_x(x-x_*,\Delta x(\omega))\Big]+ \frac12(H_{x_+(\omega)}-H_x) (x_+(\omega)-x_*)^2\nonumber\\
	\begin{split} \label{eq:letzteos}
	&= o(\|x_+(\omega)-x_*\|^2)+o(\|x-x_*\|_X^2)+o(\|x-x_*\|_X)\|\Delta x(\omega)\|_X \\
	&\quad + \frac12(H_{x_+(\omega)}-H_x) (x_+(\omega)-x_*)^2.
	\end{split}
	\end{align}
	In the last step we used second order semi-smoothness of $f$ and semi-smoothness of $f'$ at $x_*$. 
	
	We observe that the only critical term is 
	\[
	\rho \coloneqq \frac12(H_{x_+(\omega)}-H_x) (x_+(\omega)-x_*)^2. 
	\]
	We conclude
	\begin{align*}
	f(x+\Delta x(\omega)) - f(x) = f'(x)\Delta x (\omega) + \frac 12 H_x\big(\Delta x(\omega)\big)^2 + \rho + o\big(\norm{\Delta x (\omega)}{X}^2\big)
	\end{align*}
	by Corollary~\ref{cor:orelations} and then directly deduce
	\begin{align*}
	F(x_+(\omega)) - F(x) = \lambda_\omega(\Delta x(\omega))-\frac{\omega}{2}\|\Delta x(\omega)\|^2_X +\rho+o(\|\Delta x(\omega)\|_X^2) \, .
	\end{align*}
	Now, we have to consider an estimate for the critical term $\rho$ defined as above. We can define a prefactor function $\gamma: X \times [0,\infty[ \to \RR$ for the admissibility criterion \eqref{eq:lambda} by
	\begin{align*}
	\gamma(x,\omega) &\coloneqq \frac{F(x_+(\omega)) - F(x)}{\lambda_\omega(\Delta x(\omega))}\\
	&= 1+\frac{-\frac{\omega}{2}\|\Delta x(\omega)\|^2_X 
		+\rho+o(\|\Delta x(\omega)\|_X^2) }{\lambda_\omega(\Delta x(\omega))}\\
	&=1+\frac{\frac{\omega}{2}\|\Delta x(\omega)\|^2_X 
		+o(\|\Delta x(\omega)\|_X^2)-\rho}{ |\lambda_\omega(\Delta x(\omega))|}
	\end{align*}
	which should be larger than some $\tilde \gamma \in ]0,1[$. 
	We may assume that the numerator of the latter expression is non-positive, otherwise this inequality is trivially fulfilled. Thus, by decreasing the positive denominator via \eqref{eq:lambdanorm} we obtain that for any $\varepsilon>0$ there is a neigbourhood of $x^*$, such that for any iterate $x$ in this neighbourhood
	\begin{align*}
	\gamma(x,\omega) &\ge 1+\frac{\frac{\omega}{2}\|\Delta x(\omega)\|^2_X 
		-\rho+o(\|\Delta x(\omega)\|_X^2)}{\frac12(\omega+\ksum)\|\Delta x(\omega)\|_X^2}\\
	&=1+\frac{\omega}{\omega+\ksum}-\frac{(H_{x_+(\omega)}-H_x) (x_+(\omega)-x_*)^2}{(\omega+\ksum)\|\Delta x(\omega)\|_X^2}-\varepsilon
	\end{align*}
	where the latter $\varepsilon$-term arises from $o(\|\Delta x(\omega)\|_X^2)/\|\Delta x(\omega)\|_X^2$ and can be chosen arbitrarily small for $\|\Delta x(\omega)\|_X \to 0$ which holds by the estimate
	\begin{align*}
	\norm{\Delta x (\omega)}{X} \leq \norm{\Delta x}{X} \leq \norm{x_+ - x_*}{X} + \norm{x-x_*}{X} \, .
	\end{align*}
	The $\rho$-term then vanishes by assumption \eqref{eq:Hcontass}, which is implied by $i)$ or $ii)$ in the following way:
	\begin{align*}
	i) \quad \Rightarrow \quad |(H_{x_+(\omega)}-&H_x) (x_+(\omega)-x_*)^2|= |(H_{x_+}-H_x) (x_+-x_*)^2|\\
	&\le (\|H_{x_+}\|+\|H_x\|)\|x_+-x_*\|_X^2 = o\big( \norm{x-x_*}{X}^2 \big) \\
	ii) \quad \Rightarrow \quad |(H_{x_+(\omega)}-&H_x) (x_+(\omega)-x_*)^2|\\
	&\le (\|H_{x_+(\omega)}-H_{x_*}\|+\|H_{x_*}-H_x\|)\|x_+(\omega)-x_*\|_X^2 \\
	&= o\big( \norm{x-x_*}{X}^2 \big).
	\end{align*}    \qed
\end{proof}
The seemingly paradoxical behavior that full Newton steps yield a better model approximation than damped Newton steps comes from the fact that $f'$ is
not Fr\'echet differentiable in general. The only prerequisite that we can take advantage of is \eqref{eq:sos} at fixed $x_*$. 

The continuity assumption $ii)$ on $H_x$ can be verified for superposition operators via Proposition~\ref{pro:cont}, it holds, for example, for $\max(0,t)^2$, if $x_*(\omega)=0$ only on a set of zero measure.   

\section{Numerical Results} \label{sec:num}
We consider the following problem on $\Omega = [0,1]^2\subset \mathbb R^2$: Find $u \in H^1_0(\Omega,\mathbb R)$ that minimizes the composite objective functional $F$ defined via 
\begin{align} \label{eq:toymodelfunc}
F(u) \coloneqq \int_\Omega \frac 12 \norm{\nabla u}{\RR^2}^2 + \alpha \max \{ \norm{\nabla u}{\RR^2} - 1, 0 \}^2 + \beta u^3 + c\,|u| + \rho \cdot u \, \de x \, .
\end{align}
with parameters $c > 0 $ and $\alpha,\beta \in \RR$ as well as a force field $\rho : \Omega \to \RR$. The norm $\norm{\cdot}{\RR^2}$ denotes the Euclidean 2-norm on $\RR^2$. In the sense of the theory of the preceding sections we can identify the smooth part of $F$ as $f:H^1_0(\Omega,\mathbb R) \to \RR$ given by
\begin{align*}
f(u) \coloneqq \int_\Omega \frac 12 \norm{\nabla u}{\RR^2}^2 + \alpha \max \{ \norm{\nabla u}{\RR^2} - 1, 0 \}^2 + \beta u^3 + \rho \cdot u \, \de x \, .
\end{align*}
We have to note here that $f$ technically does not satisfy the assumptions made on the smooth part of the composite objective functional specified above in the case $\alpha \neq 0$ due to the lack of semi-smoothness of the corresponding squared max-term. The use of the derivative $\nabla u$ instead of function values $u$ creates a norm-gap which cannot be, as usual,  compensated by Sobolev-embeddings and hinders the proof of semi-smoothness of the respective superposition operator. However, we think that slightly going beyond the framework of theoretical results for numerical investigations can be instructive. 

For our implementation of the solution algorithm we chose the force field $\rho$ to be constant on its domain and equal to some so called load-factor $\tilde \rho > 0$ which we will from now on refer to as simply $\rho$. Consequently, the non-smooth part of the objective functional $g$ only consists of the scaled integral over the absolute value term which apparently also satisfies the specifications made on $g$ before. Note that the underlying Hilbert space is given by $X = H^1_0(\Omega,\mathbb R)$ which also determines the norm choice for regularization of the subproblem.

In the following we will dive deeper into the specifics of our implementation of the algorithm: In order to differentiate the smooth part of the composite objective functional and create a second order model of it around some current iterate, we take advantage of the automatic differentiation software package adol-C, cf. \cite{Walther2012}. With the second order model at hand we can then consider subproblem \eqref{eq:dampedstep} which has to be solved in order to obtain a candidate for the update of the current iterate. For the latter endeavor we employ a so called Truncated Non-smooth Newton Multigrid Method with a direct linear solver. We can summarize this method as a mixture of exact, non-smooth Gauß-Seidel steps for each component and global truncated Newton steps enhanced with a line-search procedure. The scheme is analytically proven to converge for convex and coercitive problems; for a more detailed description of the algorithm and its convergence properties consider \cite{Graeser2018}. 

However, the most delicate issue concerning the implementation of our algorithm and its application to the problem described above is the choice of the regularization parameter $\omega \geq 0$ along the sequence of iterates $(x_k)\subset X$. For now, we want to confine ourselves to displaying the convergence properties of the class of Proximal Newton methods in the scenario presented above and not attach too much value to algorithmic technicalities. As a consequence, we took the rather heuristic approach of simply doubling $\omega$ in the case that the sufficient descent criterion $\eqref{eq:lambda}$ (for $\gamma = \frac 12$) is not satisfied by the current update step candidate and on the other hand multiplying $\omega$ by $\big(\frac 12\big)^{n}$ where $n \in \NN$ denotes the number of consecutive accepted update steps. The latter feature ensures that local fast convergence is recognized by the algorithm and the regularization parameter quickly decreases once the iterates come close to the minimizer. For the superlinear convergence demonstrated in Theorem~\ref{thm:localacc} to arise, undamped update steps have to be conducted, i.e., the regularization parameter has to be zero and not merely sufficiently small. For this reason we set $\omega = 0$ once it reaches a threshold value $\omega_0$ following the procedure described beforehand. On the contrary, if a full update step is not accepted by the sufficient descent criterion, we set $\omega = \omega_0$ and from there on proceed as usual. 

Even though the choice of $\omega$ considered here is rather heuristic and not problem-specific at all, it stands in perfect conformity with the theory established over the course of the previous sections and also successfully displays the global convergence and local acceleration of our Proximal Newton method for the model problem of minimizing \eqref{eq:toymodelfunc} over $H^1_0(\Omega,\RR)$. Moreover, we added a threshold value for the descent considering the modified quadratic model $\lambda_\omega \big(\Delta x (\omega)\big)$ as a stopping criterion for our algorithm, i.e., the computation stops as soon as we have $|\lambda_\omega \big(\Delta x (\omega)\big)| < 10^{-14}$ for an admissible step $\Delta x (\omega)$.

Figure~\ref{fig:normcorralpha6lvl} constitutes a logarithmic plot of correction norms $\norm{\Delta x_k (\omega_k)}{H^1(\Omega)}$ for constant values of $c = 80$, $\beta = 40$ and $\rho = -100$ while $\alpha$ is increased from $0$ to $240$ in equidistant steps of $40$. Quite predictably from the structure of the functional, increasing values of $\alpha$ make the minimization problem more and more difficult to solve for our method but eventually the local superlinear convergence is evident also for larger values of $\alpha$. Figure~\ref{fig:usedregalpha6lvl} shows the corresponding values of the regularization parameter $\omega$ which were used along the accepted steps on the way to the minimizer. 

Apart from those considerations, it is always very insightful to compare the performance of our algorithm with other existing methods for similar problems to the one introduced in \eqref{eq:toymodelfunc}. To this end, we considered two alternatives: Firstly, we used a simple Proximal Gradient procedure with $H^1$-regularization by ignoring the second order bilinear form $H_x$ in the update step subproblem \eqref{eq:dampedstep} and secondly, we took advantage of acceleration strategies for such Proximal Gradient methods by implementing the FISTA-algorithm as presented in \cite{Scheinberg2014}. In Figures~\ref{fig:normcorrproxgrad} and \ref{fig:normcorrfista}, the norms of update steps are plotted for both variants for solving the same problem as above, i.e., $c = 80$, $\beta = 40$ and $\rho = -100$ while $\alpha$ we increase in equidistant steps of $40$ from $0$ to $160$. We recognize a clear difference in performance in the transition both from Proximal Gradient to FISTA and from FISTA to Proximal Newton across all $\alpha$-variations of the considered toy problem. Even in the rather mild case of $\alpha = 0$ Proximal Gradient takes $N = 5326$ and FISTA takes $N = 2498$ iterations to reach the minimizer. Note that in this case we only used four uniform grid refinements due to the very high computational effort of the simulations which does not diminish the qualitative significance of our observations.

Furthermore, Table~\ref{tab:totaliterations} displays the total number of iterations required in order to reach the minimizer of \eqref{eq:prob} considering different grid sizes for the discretization of the objective function for the values of the prefactor $\alpha$ investigated beforehand. 
In the case $\alpha > 0$ we observe some moderate increase in iteration numbers, which is attributed to the presence of a norm-gap in the corresponding term. 

\begin{figure}[h]
	\centering
	\begin{subfigure}[b]{0.45\textwidth}
		\begin{center}
			\resizebox{\textwidth}{!}{\input{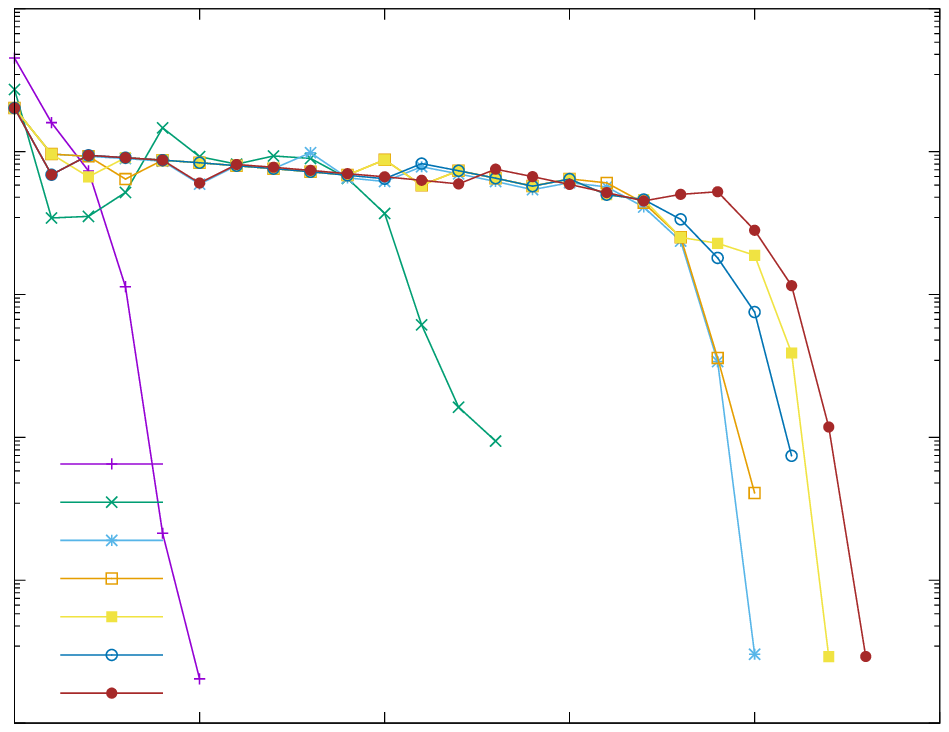}}
		\end{center}
		\caption{Correction norms $\norm{\Delta x_k (\omega_k)}{H^1(\Omega)}$}
		\label{fig:normcorralpha6lvl}
	\end{subfigure}
	\hfill
	\begin{subfigure}[b]{0.45\textwidth}
		\begin{center}
			\resizebox{\textwidth}{!}{\input{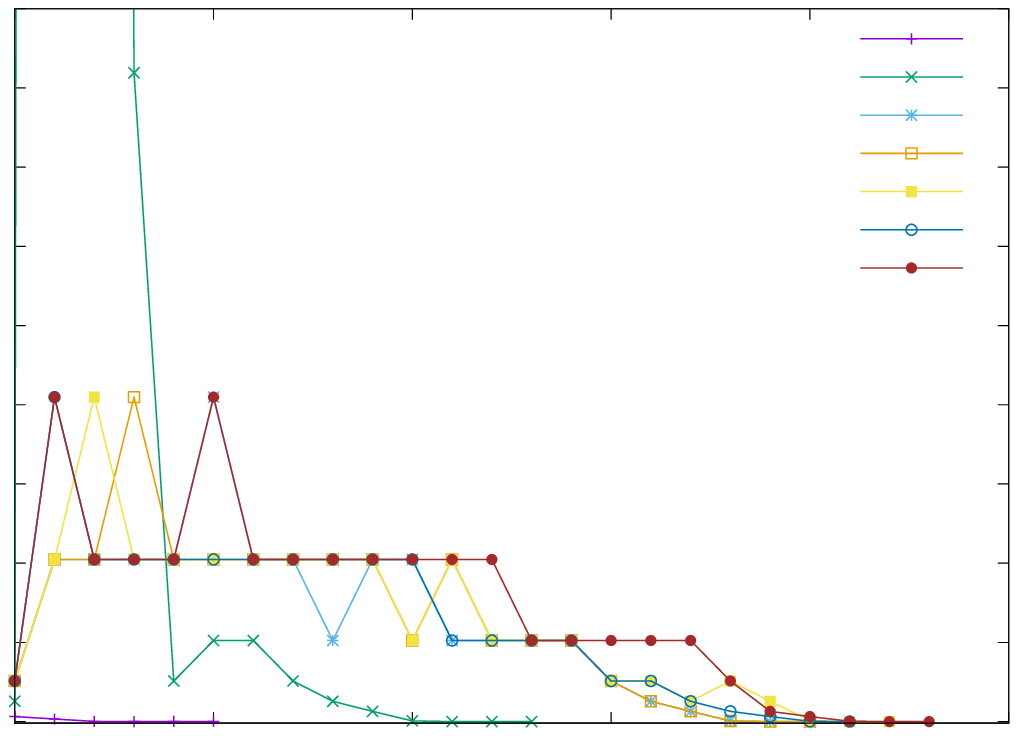}}
		\end{center}
		\caption{Regularization parameters $\omega_k$}
		\label{fig:usedregalpha6lvl}
	\end{subfigure}
	\caption{Graphs of correction norms and employed regularization parameters for $c = 80$, $\beta = 40$, $\rho = -100$ and $\alpha \in \{0,40,80,120,160,200,240\}$ for the Proximal Newton method with six uniform grid refinements.}
\end{figure}

\begin{figure}[h]
	\centering
	\begin{subfigure}[b]{0.45\textwidth}
		\begin{center}
			\resizebox{\textwidth}{!}{\input{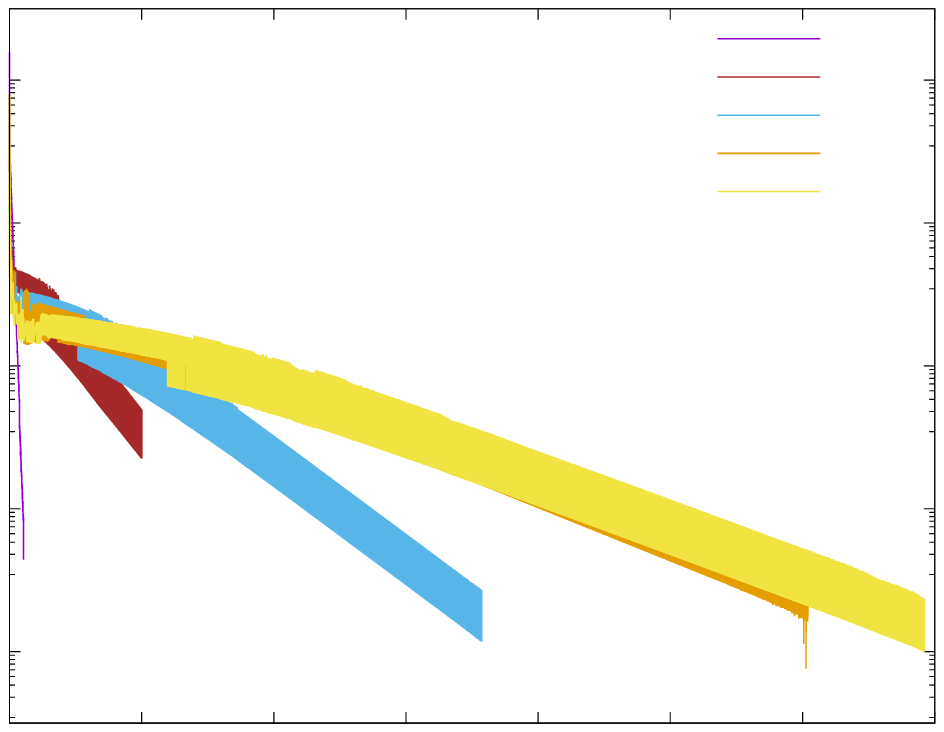}}
		\end{center}
		\caption{Proximal Gradient method}
		\label{fig:normcorrproxgrad}
	\end{subfigure}
	\hfill
	\begin{subfigure}[b]{0.45\textwidth}
		\begin{center}
			\resizebox{\textwidth}{!}{\input{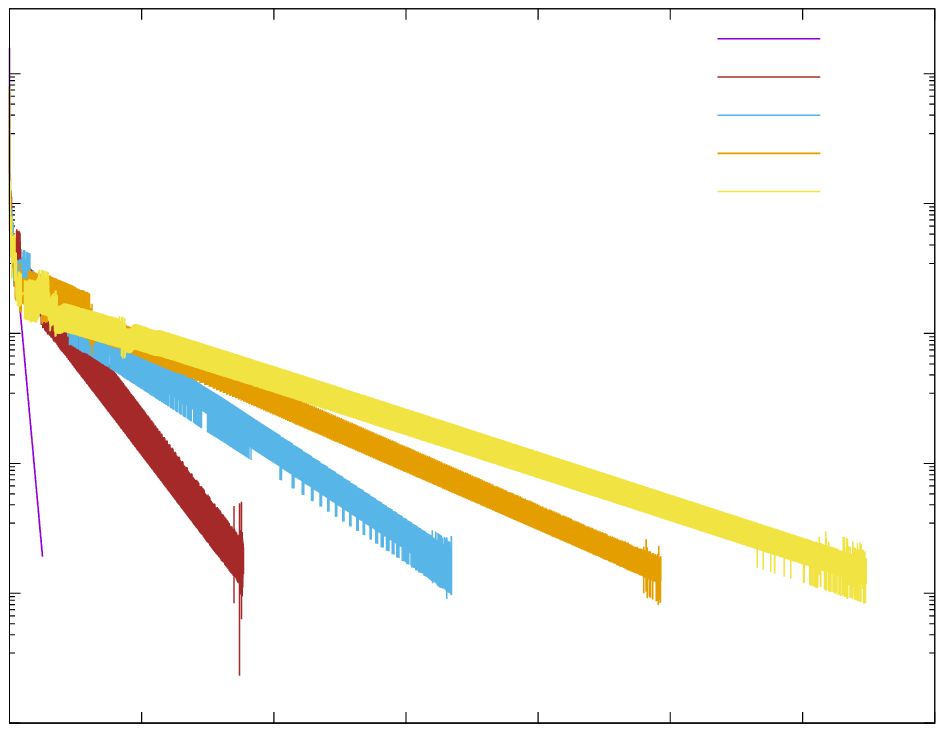}}
		\end{center}
		\caption{FISTA method}
		\label{fig:normcorrfista}
	\end{subfigure}
	\caption{Graphs of correction norms for Proximal Gradient and FISTA with $c = 80$, $\beta = 40$, $\rho = -100$ and $\alpha \in \{0,40,80,120,160\}$ with four uniform grid refinements.}
\end{figure}

%

\begin{table}[h]
	\begin{center}
		\begin{tabular}{| c || c | c | c | c | c | c | c |}
			\hline
			\xrowht{20pt} \diagbox{$h$}{$\alpha$} & 0 & 40 & 80 & 120 & 160 & 200 & 240 \\
			\hline \hline
			\xrowht{20pt} $2^{-4}$ & 5 & 9 & 9 & 12 & 13 & 13 & 14 \\
			\hline
			\xrowht{20pt} $2^{-5}$ & 5 & 11 & 15 & 20 & 19 & 22 & 25  \\
			\hline
			\xrowht{20pt} $2^{-6}$ & 6 & 14 & 21 & 21 & 23 & 22 & 24  \\
			\hline
			\xrowht{20pt} $2^{-7}$ & 5 & 19 & 21 & 28 & 28 & 32 & 38  \\
			\hline
			\xrowht{20pt} $2^{-8}$ & 7 & 25 & 26 & 28 & 30 & 37 & 39  \\
			\hline
		\end{tabular}
	\end{center}
	\caption{Number of total iterations $N$ for different grid sizes $h$ and prefactor values $\alpha$ for fixed parameters $\beta=40$ and $c=80$.}
	\label{tab:totaliterations}
\end{table}

\section{Conclusion} \label{sec:conc}

Now that we have sufficiently displayed the global and local convergence properties of our Proximal Newton method, it is time to both reflect on what we have achieved here as well as discuss some possible improvements on the algorithm and its implementation which are a topic of future research:

We have developed a globally convergent and locally accelerated Proximal Newton method in a Hilbert space setting which demands neither second order differentiability of the smooth part nor convexity of either part of the composite objective function. Concerning differentiability, we have introduced the notion of second order semi-smoothness. Concerning non-convexity, our theoretical framework uses quantified information on lacking convexity instead of simply resorting to a different first order update scheme in the non-convex case. The globalization scheme takes advantage of a proximal arc search procedure and thereby establishes stationarity of all limit points of the sequence of iterates. Additional convexity close to optimal solutions of the original problem leads to local acceleration of our method which in particular does not rely on strong convexity of the smooth part, but only on the strong convexity of the composite functional thanks to a well-thought definition of proximal mappings within the theoretical framework. The application of our method to actual function space problems is enabled by using an efficient solver for the step computation subproblem, the Truncated Non-smooth Newton Multigrid Method. We have displayed global convergence and local acceleration of our algorithm by considering a toy model problem in function space.

As we have already mentioned beforehand, the choice of the regularization parameter we employed here is rather heuristic and not problem-specific at all. This issue can be addressed by using an estimate for the residual term of the quadratic model established in subproblem \eqref{eq:dampedstep}, as seen in \cite{Weiser2007} for adaptive affine conjugate Newton methods where non-convex but smooth minimization problems for nonlinear elastomechanics have been thoroughly investigated. The idea behind the procedure is to evaluate actual residual terms for formerly computed correction candidates and then use them as a regularization parameter for the computation of the next update step candidate.

Another focal concern of our future work is taking into account inexactness in the computation of update steps. Inexact solutions of subproblem \eqref{eq:dampedstep} are then required to at least satisfy certain inexactness criteria which still give access to similar global and local convergence properties of the ensuing algorithm as the exact version discussed throughout the present treatise. 

Additionally, these inexactness criteria should be sufficiently simple to evaluate since they have to be considered within every iteration of solving the subproblem for update step computation. However, the discussion of inexact Proximal Newton methods then opens up the possibility of considering more challenging real-world applications like energetic formulations of finite strain plasticity.

\bibliographystyle{spmpsci}
\bibliography{References.bib}
\end{document}